\documentclass{amsart}
\usepackage{amsmath,amsthm}
\usepackage{amsfonts,amssymb}
\usepackage{enumerate}
\usepackage{accents,color}
\usepackage{graphicx}
\usepackage{bm}
\usepackage{algorithm}
\usepackage{algpseudocode}

\newcommand{\lvt}{\left|\kern-1.35pt\left|\kern-1.3pt\left|}
\newcommand{\rvt}{\right|\kern-1.3pt\right|\kern-1.35pt\right|}

\newtheorem{thm}{Theorem}[section]

\newtheorem{prop}[thm]{Proposition}

\theoremstyle{remark}

 \def\a{{\alpha}}
 \def\b{{\beta}}
 \def\g{{\gamma}}

 \def\d{\mathrm{d}}

 \def\la{{\langle}}
 \def\ra{{\rangle}}

 \def\CB{{\mathcal B}}

 \def\CV{{\mathcal V}}

 \def\CC{{\mathbb C}}
 
 \def\NN{{\mathbb N}}
 \def\PP{{\mathbb P}}
 
 \def\RR{{\mathbb R}}
 
 \def\YY{{\mathbb Y}}

\graphicspath{{figures./}}

\begin{document}

\title{Orthogonal polynomials on a class of planar algebraic curves}

\author{Marco Fasondini}
\address{School of Computing and Mathematical Sciences\\
University of Leicester\\
 United Kingdom  }\email{m.fasondini@leicester.ac.uk}

\author{Sheehan Olver}
\address{Department of Mathematics\\
Imperial College\\
 London \\
 United Kingdom  }\email{s.olver@imperial.ac.uk}

\author{Yuan Xu}
\address{Department of Mathematics\\ University of Oregon\\
    Eugene, Oregon 97403-1222.}\email{yuan@uoregon.edu}

\date{\today}

\begin{abstract}
We construct bivariate orthogonal polynomials (OPs) on algebraic curves of the form $y^{m} = \phi(x)$ in $\RR^2$ where $m = 1, 2$ and $\phi$ is a polynomial of arbitrary degree $d$, in terms of univariate semiclassical OPs. We compute connection coefficients that relate the bivariate OPs to a polynomial basis that is itself orthogonal and whose span contains the OPs as a subspace. 
The connection matrix is shown to be banded and the connection coefficients and Jacobi matrices for OPs of degree $0, \ldots, N$ are computed via the Lanczos algorithm in $\mathcal{O}(Nd^4)$ operations. 
\end{abstract}
\maketitle

\section{Introduction}
\setcounter{equation}{0}

In previous works, we studied orthogonal polynomials (OPs) on planar curves and on quadratic surfaces of revolution. The OPs were constructed explicitly on the wedge and the square~\cite{OX1}, planar quadratic curves~\cite{OX2}, cubic curves~\cite{FOX}, as well as on and inside quadratic surfaces of revolution~\cite{OX3}.  These OPs were used in applications, such as the statistics of determinantal point processes, the computation  of Stieltjes transforms, multivariate function approximation and orthogonal series~\cite{Xu1,Xu2,Xu3,Xu4,Xu5,Xu6},  the approximation of singular or nearly singular functions of one variable (including those that arise as solutions to differential equations), the explicit solution to wave equations~\cite{OX4}, spectral methods for partial differential equations on trapeziums, disk slices and spherical caps~\cite{SO1,SO2} and a spectral method for fractional integral equations~\cite{pu}. 

In this paper we generalise both the types of curves on which OPs are constructed as well as their method of construction, in particular to OPs defined on curves of the form $y^{m} = \phi(x)$ in $\RR^2$ where $m = 1, 2$ and $\phi$ is a polynomial of arbitrary degree $d$. Unlike the previously mentioned studies, our constructions of OPs are computational, not explicit. However, on quadratic and cubic curves, our computational method reproduces the OP bases given in~\cite{OX2,FOX} and in section~\ref{sec:exOPs} we give explicit bases for a few new special cases. Our approach is first to  explicitly construct a polynomial basis whose span contains the bivariate OPs as a subspace, then we cast the computation of the OPs as a numerical linear algebra problem, namely, simultaneously block-tridiagonalize two symmetric, commuting matrices, which we solve with  the Lanczos algorithm. This is a very general approach that allows us to construct OPs (i) on certain algebraic curves (the topic of this paper) and also (ii) inside those curves (two-dimensional domains), (iii) their surfaces of revolution (three-dimensional regions) and (iv) associated star domains (2D and 3D regions). The construction of OPs on  (ii)--(iv) will be considered elsewhere. Sparse spectral methods for PDEs is an application of the OPs we construct that is of special interest and will be pursued in future research.


\noindent {\bf Acknowledgment}. The first and second authors were supported by the Leverhulme Trust Research Project Grant RPG-2019-144 ``Constructive approximation theory on and inside algebraic curves and surfaces''. The second author was also supported by the Engineering and
Physical Sciences Research Council (EPSRC) Standard Grant EP/T022132/1. The third author was partially supported by Simons Foundation Grant \#849676.


\section{Univariate OPs} 
\setcounter{equation}{0}
Throughout, we shall denote a univariate, orthonormal, \mbox{degree-$n$} OP that is orthogonal  with respect to the inner product
\begin{equation}
\la f, g\ra_{\omega} = \int f(x)g(x)\omega(x)\d x, \label{eq:1dip}
\end{equation}
where $\omega$ is a non-negative weight function,  by $p_n(\omega ; x)$, or $p_n(\omega)$. We shall also make use of quasimatrix~\cite{Tref} notation. For example, the quasimatrix
\begin{align}
\mathbf{P}(\phi w)  = \left(
\begin{array}{c c c c}
p_0(\phi w) & p_1(\phi w) & p_2(\phi w) & \cdots
\end{array}
\right),  \label{eq:qmex}
\end{align}
where $\phi = \phi(x)$ is polynomial that is strictly positive on the support of the weight $w$, is a `matrix' with a countable infinity of columns consisting of OPs with respect to the weight $\phi w$, ordered by degrees. 

We shall let $w$ be a classical weight, i.e.,  a Jacobi, Laguerre or Hermite weight~\cite[Ch.~18]{NIST:DLMF} with support $\mathrm{supp}(w) = (-1, 1), (0, \infty)$ and $(-\infty, \infty)$, respectively. If all the roots of the polynomial $\phi$ (with $\phi(x)>0$ for $x \in \mathrm{supp}(w)$) are at the endpoints of $\mathrm{supp}(w)$, then $\phi w$ is also a classical weight, otherwise at least one of the roots of $\phi$ are off $\mathrm{supp}(w)$ and $\phi w$ is a semiclassical weight~\cite{hrsemiclass} and we refer to the $p_n(\phi w)$ as (orthonormal) semiclassical OPs. 

The fundamental building blocks of our constructions of bivariate OPs will be Jacobi matrices and a `raising matrix' of univariate OPs, which we shall now define. Recall that if a family of OPs, e.g., $p_n(\phi w)$, $n \geq 0$,  is orthonormal with respect to an inner product $\la \cdot, \cdot \ra_{\phi w}$, then they satisfy a three-term recurrence of the form
\begin{align}
xp_n(\phi w) = \beta_{n-1} p_{n-1}(\phi w) + \alpha_{n} p_{n}(\phi w) + \beta_{n} p_{n+1}(\phi w), \qquad n \geq 0, \label{eq:sc3term}
\end{align}
where $\beta_{-1} = 0 = p_{-1}(\phi w)$ and the recurrence coefficients are
\begin{equation}
\alpha_{n} = \alpha_{n}(\phi w) = \la xp_n(\phi w), p_n(\phi w) \ra_{\phi w}, \qquad \beta_n = \beta_n(\phi w) =  \la xp_n(\phi w), p_{n+1}(\phi w) \ra_{\phi w}.  \label{eq:screccoeffs}
\end{equation}
Since $w$ is a classical weight and $\phi(x) > 0$ for $x \in \mathrm{supp}(w)$, $\alpha_n(\phi w), \beta_n(\phi w) \in \mathbb{R}$ for $n \geq 0$.
 In quasimatrix notation, (\ref{eq:sc3term}) becomes
\begin{equation}
x\mathbf{P}(\phi w) = \mathbf{P}(\phi w)J(\phi w), \label{eq:qmjacop}
\end{equation}  
where $J(\phi w)$ is an infinite, symmetric, tridiagonal matrix known as the Jacobi matrix of the OPs,
\begin{equation}
J(\phi w) = \left(
\begin{array}{c c c c}
\alpha_0(\phi w) & \beta_0(\phi w)  & & \\
\beta_0(\phi w)  & \alpha_1(\phi w) & \beta_1(\phi w) & \\
         & \beta_1(\phi w)  & \alpha_2(\phi w)  & \ddots \\
         &          & \ddots  & \ddots
\end{array}
\right).  \label{eq:jacphiw}
\end{equation}

We say that a (finite or infinite) matrix  $A$ is banded and has bandwidths $(\lambda, \mu)$, $\lambda, \mu \geq 0$ if $A_{i,j} = 0$ for $i-j>\lambda$ and $j-i> \mu$ and the same applies for block-matrices (in which case $A_{i,j}$ refers to block $(i,j)$ of $A$). Hence, $J(\phi w)$ has bandwidths $(1, 1)$.
\begin{prop}
If $\deg \phi = d$, then the OPs $\{p_n(w) \}$ and $\{ p_n(\phi w) \}$ are related via a banded raising matrix $R$ with bandwidths $(0,d)$ according to
\begin{equation}
\mathbf{P}(w) = \mathbf{P}(\phi w)R,  \label{eq:raiseop}
\end{equation}
where $\mathbf{P}(\phi w)$ is given in (\ref{eq:qmex}) and  $\mathbf{P}(w)$ is defined analogously and
\begin{equation}
R = \left(
\begin{array}{c c c c c c c}
r_{0,0} &  \cdots & \cdots  & r_{0,d}   &   &   & \\
        & r_{1,1} & \cdots   & \cdots & r_{1,d+1} & & \\
        &         & r_{2,2}   &  \cdots & \cdots & r_{2,d+2} &   \\
        &         &           & \ddots   &   \ddots  & \ddots  & \ddots
\end{array}
\right),  \label{eq:raisestruct}
\end{equation}
where
\begin{equation}
r_{k,n} = \langle p_n(w), p_k(\phi w) \rangle_{\phi w} = \langle p_n(w), \phi p_k(\phi w) \rangle_{ w}.  \label{eq:rkndef}
\end{equation}
\end{prop}
\begin{proof}
Since any polynomial of degree $n$ is a linear combination of the $p_k(\phi w)$ with $0 \leq k \leq n$, we have $p_n(w) = \sum_{k = 0}^n r_{k,n} p_{k}(\phi w)$, where the expression (\ref{eq:rkndef}) and the fact that $r_{k,n} = 0$ for $0 \leq k < n-d$ follow since the $p_k(\phi w)$ and $p_n(w)$ are (orthonormal) OPs. 
\end{proof}

In quasimatrix notation, we write
\begin{equation}
R = \langle \mathbf{P}(\phi w), \mathbf{P}( w) \rangle_{\phi w} =  \int  \mathbf{P}^{\top}(\phi w) \mathbf{P}( w) \phi w \, \mathrm{d} x, \label{eq:qmipd}
\end{equation}
where $ \mathbf{P}^{\top}(\phi w) \mathbf{P}( w)$
is interpreted as an outer product of two (infinite) row vectors whose entries are functions and integration is performed entry-wise. Therefore,  $\langle \mathbf{P}(\phi w), \mathbf{P}( w) \rangle_{\phi w}$ is an infinite matrix whose entries are defined as follows: Let $\bm{e}_k$ be the standard basis vector
\begin{equation}
\bm{e}_k = \Big( \underbrace{0 \: \cdots \: 0}_{k \text{ zeros}} \: 1 \: 0 \: 0 \: \cdots \Big)^{\top} \label{eq:basisvec}
\end{equation}
then the $(k,n)$ entry of  $\langle \mathbf{P}(\phi w), \mathbf{P}( w) \rangle_{\phi w}$ for $k, n \geq 0$ is (cf.~(\ref{eq:rkndef}))
\begin{equation}
\left(\langle \mathbf{P}(\phi w), \mathbf{P}( w) \rangle_{\phi w}\right)_{k,n} =  \int  \mathbf{P}(\phi w)\bm{e}_k \mathbf{P}( w)\bm{e}_{n} \phi w \, d x = \langle p_n(w), p_k(\phi w) \rangle_{\phi w} = r_{k,n}.  \label{eq:qmipent}
\end{equation}
Similarly, we can express the Jacobi matrix as $J(\phi w) = \langle \mathbf{P}(\phi w), x\mathbf{P}(\phi w) \rangle_{\phi w}$ and since the OPs $\{p_n(w) \}$ and $\{ p_n(\phi w) \}$ are orthonormal, $\langle \mathbf{P}(\phi w), \mathbf{P}(\phi w) \rangle_{\phi w} = I = \langle \mathbf{P}( w), \mathbf{P}( w) \rangle_{ w}$, where $I$ is an infinite identity matrix.

The coefficients $r_{k,n}$ of the raising matrix $R$ can also be viewed as the connection coefficients relating the classical OP basis $\{p_n(w) \}$ to the semiclassical OP basis $\{p_n(\phi w) \}$. Algorithms for computing connection coefficients between (mainly univariate) classical OP families with different weights are discussed in~\cite{olver2020fast}. The more general problem of relating the OP families $\{ p_n(r w) \}$ and $\{ p_n( w) \}$, where $r = r(x)$ is a rational function, and determining the recurrence coefficients of the former OP family is discussed in~\cite[sect.~2.4]{Gautschi} and associated algorithms are considered in~\cite[Ch.~5]{golub}. For our purposes (i.e., the case $r(x) = \phi(x)$), we note that one method for computing $R$ and $J(\phi w)$ is via the Cholesky factor $L$ of the symmetric, positive definite matrix
 \begin{equation*}
\Phi = \langle \mathbf{P}(w), \mathbf{P}(w)\rangle_{\phi w} = \int \mathbf{P}(w)^{\top} \mathbf{P}(w) \phi w \,\mathrm{d}x = LL^{\top}
\end{equation*}
because from (\ref{eq:raiseop}) we have that
\begin{equation*}
I = \langle \mathbf{P}(\phi w), \mathbf{P}(\phi w)\rangle_{ w} = \int \mathbf{P}(\phi w)^{\top} \mathbf{P}(\phi w)  \phi w \,\mathrm{d}x = R^{-\top}\Phi R^{-1}  = R^{-\top}LL^{\top} R^{-1}
\end{equation*}
from which we conclude that
\begin{equation}
R = L^{\top}, \qquad J(\phi w) =R^{\top} J(w) R,  \label{eq:Rchol}
\end{equation}
where the second identity follows from (\ref{eq:qmjacop}) and (\ref{eq:raiseop}).  Another method for computing $R$ and $J(\phi w)$ given $J(w)$ is to use Christoffel's theorem~\cite[sect.~2.4]{Gautschi}. 


 
The topic of this paper is also the computation of connection coefficients (and, as a byproduct, recurrence coefficients will also be computed), however we shall relate bivariate OPs (restricted to curves) to a polynomial basis that is orthogonal with respect to the \emph{same} weight as the bivariate OPs. As we shall see, not only is our connection problem complicated by the fact that it is bivariate but also because in general the degree of a bivariate polynomial on $\mathbb{R}^2$ is different from its degree when restricted to an algebraic curve in  $\mathbb{R}^2$.


\section{Bivariate OPs on $y = \phi$ and $y^2 = \phi$}
\setcounter{equation}{0}

Curves defined by $p(x,y) = 0$, where $p$ is an irreducible bivariate quadratic or cubic polynomial, can be simplified to the form $y^2 = \phi(x)$, with $\deg \phi \leq 2$ for quadratics and  $\deg \phi = 3$ for cubics, see \cite{OX2} and~\cite{FOX}, respectively. Through affine transformations, non-degenerate quadratics can be reduced to either a parabola ($\phi = x$), a circle ($\phi = 1 - x^2$), or a hyperbola ($\phi = x^2 - 1$). A cubic curve in the standard form $y^2 = \phi(x)$, $\deg \phi = 3$ is an elliptic curve if it has no cusps, isolated points or self-intersections; the OPs we constructed in~\cite{FOX} are valid on both elliptic and non-elliptic curves.  In this paper we  consider bivariate OPs on the following more general curves
\begin{align}
\g_{m,d} = \left\{ (x,y) \in \RR^2 :  \begin{cases}
y = \phi(x) &\text{if } m = 1 \\
y^2 = \phi(x) > 0 &\text{if } m = 2
\end{cases}, x \in \mathrm{supp}(w), \deg \phi = d > 0
\right\} \label{eq:mcurvedef} 
\end{align}
where $\mathrm{supp}(w)$ denotes the support of a non-negative weight function $w(x)$; either $\mathrm{supp}(w) = (a, b)$, $(a, \infty)$ or $(-\infty, \infty)$, where $-\infty < a < b <\infty$ and $\phi$ may vanish at $a$ or $b$. The restriction $\phi>0$ does not apply for $m = 1$ but  for $m=2$, the positivity of $\phi$ is necessary for the curve to be defined in $\mathbb{R}^2$.

For the special case $m = 1$, $d = 0$ and a classical weight $w$, the classical OPs are OPs on the curves $\g_{1,0}$. The case $m = 2$ and $d = 0$ is treated in \cite{OX2}. OPs on curves of the form $y^m = \phi$ with $m >2$, which  we leave for future work, will be discussed briefly in section~\ref{sect:conc}.

We consider OPs in $(x, y)$ that are restricted to the curve $\g_{m,d}$ and which are orthogonal with respect to the following inner product
$$
\la f, g \ra_{\g_{m,d},w} = \int_{\g_{m,d}} f(x,y) g(x,y) \omega(x,y)\mathrm{d} \sigma(x,y),
$$
where $\mathrm{d}\sigma$ is the arc length measure on $\g_{m,d}$. We incorporate the arc length into the non-negative weight function by setting $\omega(x,y) = w(x)/\sqrt{1 + (y'(x))^2}$ so that the inner product simplifies to
\begin{align} \label{eq:mipd}
  \la f,g \ra_{\g_{m,d},w} = \begin{cases}  
&     \displaystyle{\int}  f\left(x, \phi(x) \right) g\left(x, \phi(x) \right)  w(x) \d x,\hspace{2.15 cm} m = 1, \\  
  & \displaystyle{\int} \left[  f\left(x, \sqrt{\phi(x)} \right) g\left(x, \sqrt{\phi(x)} \right) \right.  \\
      & \left.\:\:\: + f\left(x, - \sqrt{\phi(x)} \right) g\left(x, -\sqrt{\phi(x)} \right)\right] w(x) \d x,\:\:\: m = 2, 
      \end{cases}
\end{align} 
where the integrals are over $\mathrm{supp}(w)$.  The bilinear form $\la\cdot, \cdot  \ra_{\g_{m,d}, w}$ defines an inner product on the space of polynomials of two variables restricted to the curve $\g_{m,d}$, 
\begin{equation*}
      \Pi(\g_{m,d}) := \RR[x,y]/ \la y^m - \phi (x)\ra,     
\end{equation*}
which denotes the space of bivariate polynomials modulo the ideal generated by $y^m - \phi(x) = 0$.

\begin{prop}\label{prop:orthdimmd}
Let $\CV_n(\g_{m,d}, w)$ (or $\CV_n$) be the space of bivariate OPs of degree $n$ with respect to the inner product (\ref{eq:mipd}) defined on $\Pi(\g_{m,d})$, then
\begin{equation}
\dim \CV_n(\g_{m,d}, w) = \min\{n+1, \max\{ d, m\} \}.  \label{eq:cvn}
\end{equation}
\end{prop}
\begin{proof}
We first consider the dimension of the space of monomials of degree $n$ on the curve $\g_{m,d}$, which we denote by $\CB_n(\g_{m,d})$. Three cases arise: (i) $n \leq \max\{d, m\} - 1$, (ii) $m > d, n \geq m$ and (iii) $m \leq d$, $n \geq d$. Since we only let $m = 1, 2$ and $d > 0$, case (ii) only arises if $m = 2$ and $d = 1$. 

For case(i), we have that
\begin{equation}
\CB_n(\g_{m,d}) = \displaystyle{\left\{ x^{n-k}y^{k} \right\}_{k=0}^{n}}, \qquad n \leq \max\{d, m\} - 1,  \label{eq:Bgc1}
\end{equation}
because the monomials $x^ky^{n-k}$ are linearly independent for $0 \leq k \leq n$ and thus $\dim \CB_n(\g_{m,d}) = n+1$. 

For  case (ii),
\begin{equation}
\CB_n(\g_{m,d}) = \left\{ x^{n-k}y^{k} \right\}_{k=0}^{m-1}, \qquad m > d, \qquad n \geq m, \label{eq:Bgc2}
\end{equation}
because we can set $y^m = \phi(x)$ and express the monomials $x^{n-k}y^{k}$ with $k \geq m$  in terms of the monomials in $\CB_k(\g_{m,d})$, with $0 \leq k \leq n$, hence only the monomials $x^{n-k}y^{k}$ with $0 \leq k \leq m-1$  are linearly independent and thus  $\dim \CB_n(\g_{m,d}) = m$. 

For the final case (iii),
\begin{equation}
\CB_n(\g_{m,d}) = \left\{ x^{d-1}y^{n+1-d}, x^{d-2}y^{n+2-d}, \ldots, y^n  \right\}, \qquad d \geq m, \qquad n \geq d, \label{eq:Bgc3}
\end{equation}
because we can use the fact that $y^m = c_dx^d + \cdots + c_1 x + c_0$, with $c_d \neq 0$ to set $x^d = \psi(x,y)$, where $\deg \psi < d$, and express $y^{n-k}x^{k}$ with $k \geq  d$ in terms of lower degree monomials, therefore $\dim \CB_n(\g_{m,d}) = d$

Applying the Gram--Schmidt process to the basis $\CB_n(\g_{m,d})$ in cases (i)--(iii), it follows that $\dim \CB_n(\g_{m,d}) = \dim \CV_n(\g_{m,d}, w)$, which completes the proof.
\end{proof}
\begin{prop}\label{prop:curvepolyform}
A bivariate polynomial $p(x, y) \in  \Pi(\g_{m,d})$ can be represented in terms of polynomials of one variable as follows,
\begin{equation}
p(x,y) = \begin{cases}
q(x) & \text{if } m = 1, \\
r(x) + ys(x) & \text{if } m = 2,
\end{cases}  \label{eq:maxdr}
\end{equation}
where $q(x), r(x), s(x)$ are polynomials in $x$.
\end{prop}
\begin{proof}
Any  $p(x, y) \in  \Pi(\g_{m,d})$ can be expressed as a linear combination of the elements of the monomial bases in (\ref{eq:Bgc1})--(\ref{eq:Bgc3}). The result follows immediately upon setting $y = \phi(x)$ (if $m = 1$) or $y^2 = \phi(x)$ (if $m = 2$) in the monomial expansion of $p(x,y)$. 
\end{proof}
In view of this result, we introduce the ``$P$-polynomials'' or ``$P$-basis''
\begin{align}
P^{(m,d)}_{n,k}(x,y) = \begin{cases}
 p_n(w;x) & \text{for } m = 1, k = 0, n \geq 0 \\
 p_n(w;x)/\sqrt{2} & \text{for } m = 2, k = 0, n \geq 0 \\
 yp_{n-1}(\phi w;x)/\sqrt{2} & \text{for } m = 2, k = 1, n \geq 1
\end{cases},   \label{eq:Pmddef}
\end{align}
which clearly form a basis for $\Pi(\g_{m,d})$ and, moreover:
\begin{prop}\label{prop:Pmorthog}
The $P$-polynomials are orthonormal with respect to $\langle\cdot, \cdot \rangle_{\g_{m,d},w}$:
\begin{equation*}
\langle P^{(m,d)}_{n,k}, P^{(m,d)}_{\ell,i} \rangle_{\g_{m,d},w} = \delta_{n,\ell}\delta_{k,i}, \qquad m = 1, 2.
\end{equation*}
\end{prop}
\begin{proof}
Since the $p_n(w)$ are orthonormal with respect to the univariate inner product $\langle \cdot, \cdot \rangle_{w}$ we have, for $m = 1$, 
\begin{equation*}
\langle P^{(1,d)}_{n,0}, P^{(1,d)}_{\ell,0} \rangle_{\g_{1,d},w} = \langle p_n(w), p_{\ell}(w) \rangle_{w} = \delta_{n,\ell}. 
\end{equation*}
For $m = 2$, three orthogonality relations need to be demonstrated. First, orthogonality between $P$-polynomials with $k = 0$ in (\ref{eq:Pmddef}):
\begin{equation*}
\langle P^{(2,d)}_{n,0}, P^{(2,d)}_{\ell,0} \rangle_{\g_{2,d},w} =  \langle p_n(w), p_{\ell}(w) \rangle_{w} = \delta_{n,\ell}.
\end{equation*}
Next, between polynomials with $k = 1$,
\begin{equation*}
\langle P^{(2,d)}_{n,1}, P^{(2,d)}_{\ell,1} \rangle_{\g_{2,d},w} = \langle \phi p_{n-1}(\phi w), p_{\ell-1}(\phi w) \rangle_{w} = \langle  p_{n-1}(\phi w), p_{\ell-1}(\phi w) \rangle_{\phi w} =  \delta_{n,\ell}. 
\end{equation*} 
The final orthogonality relation is due to the reflection symmetry of the curve for $m = 2$:
\begin{equation*}
\langle P^{(2,d)}_{n,1}, P^{(2,d)}_{\ell,0} \rangle_{\g_{2,d},w} = \int\left[yp_{n-1}(\phi w)p_{\ell}(w) -  yp_{n-1}(\phi w)p_{\ell}(w)  \right]w \d x = 0.
\end{equation*}
\end{proof}

\subsection{Connection coefficients}
We denote the degree-$n$, orthonormal OPs on $\g_{m,d}$ by $Y^{(m,d)}_{n,k}(x,y)$, where $k = 1, \ldots, \dim \CV_n$, with $\dim \CV_n$  defined in (\ref{eq:cvn}). Hence, the OPs satisfy
\begin{equation*}
\langle Y^{(m,d)}_{n,k}, Y^{(m,d)}_{\ell,i} \rangle_{\g_{m,d},w} = \delta_{n,\ell}\delta_{k,i}, \qquad m = 1, 2,
\end{equation*}
for $n, \ell \geq 0$,  $1 \leq k \leq \dim\CV_k$, $1 \leq i \leq \dim\CV_i$ and $Y^{(m,d)}_{n,k}$ has degree $n$ on $\g_{m,d}$.  

It follows from Proposition~\ref{prop:curvepolyform} that the OPs can be expressed as linear combinations of the $P$-polynomials. That is, there exist \emph{connection coefficients} $C^{(m,d)}_{j,\ell}$, where $\ell = \ell(n,k) \in \mathbb{N}_0$ will be defined in section~\ref{sect:rec4ms}, such that
\begin{align}
 Y_{n,k}^{(1,d)}(x,y) & = \sum_{j \geq 1}   C^{(1,d)}_{j,\ell}P_{j,0}^{(1,d)}, \label{eq:Y1exp} \\
 Y_{n,k}^{(2,d)}(x,y) & = \sum_{j \geq 1} C^{(2,d)}_{2j-1,\ell} P_{j,0}^{(2,d)} +  \sum_{j \geq 1} C^{(2,d)}_{2j,\ell} P_{j,1}^{(2,d)}, \quad \label{eq:Y2exp}
\end{align}
for $n \geq 1$, $k = 1, \ldots, \dim \CV_n$ and with $Y_{0,1}^{(m,d)} = C^{(m,d)}_{0,0}P_{0,0}^{(m,d)}$ for $m = 1, 2$. The upper and lower bounds of the summation indices in (\ref{eq:Y1exp})--(\ref{eq:Y2exp}) will be considered in section~\ref{sect:Cstruct}. 

\subsection{Minimal and maximal degrees of polynomials on $\g_{m,d}$}

If $p$ is a polynomial in $\mathbb{R}[x,y]$, then we define the degree of the restriction of $p$ to $\g_{m,d}$ as the lowest degree of $q$ among all $q$ satisfying $p = q$ modulo $\langle y^m - \phi(x) \rangle$. By construction, the monomial bases in (\ref{eq:Bgc1})--(\ref{eq:Bgc3}) are minimal degree representations of polynomials in $\Pi(\g_{m,d})$. Therefore, the degree of $p(x,y)$ on the curve $\g_{m,d}$ is the degree of its expansion in the monomial basis given in (\ref{eq:Bgc1})--(\ref{eq:Bgc3}). 

If $d  > m$ and  $p(x,y) \in \Pi(\g_{m,n})$, then a maximal degree representation of $p$ is obtained by  setting $y = \phi$ (if $m=1$) or $y^2 = \phi$ (if $m=2$), in which case $p$ takes the form given in (\ref{eq:maxdr}). For example, if $m = 1 < d$, then $y^n = \phi^n$ is a polynomial of degree $n$ in $\Pi(\g_{m,n})$ whose maximal degree on $\g_{1,d}$ is $nd$. Observe from (\ref{eq:Pmddef}) that the total degree of a $P$-polynomial $P^{(m,d)}_{n,k}$ is $n$, which is also its maximal degree since it is in the form (\ref{eq:maxdr}). Therefore, the maximal degree of $p(x,y) \in \Pi(\g_{m,n})$ is the total degree of its expansion in the $P$-basis. We denote the maximal degree of $p(x,y)$ on $\g_{m,d}$ by $\deg p$ and its minimal degree by $\deg_{\g_{m,d}}(p)$. For example, if $m = 1 < d$, then $\deg y^n = nd$ and $\deg_{\g_{1,d}}(y^n) = n$ and for $n > d$, $\deg P^{(m,d)}_{n,k} = n > \deg_{\g_{m,d}}(P^{(m,d)}_{n,k})$ while, by definition, $\deg_{\g_{m,d}}(Y^{(m,d)}_{n,k}) = n$.

%

%
%
%

\section{Curves with explicit OP bases}\label{sec:exOPs}
\setcounter{equation}{0}

Before recursively constructing OPs on the general curves (\ref{eq:mcurvedef}), we consider cases for which the OP basis can be given explicitly in terms of the $P$-polynomials.

\begin{prop} \label{prop:exOPs}
The following are orthonormal OP bases on curves $\g_{m,d}$ defined in (\ref{eq:mcurvedef}):
if $m = 1$ and $\deg \phi = 1$, then
\begin{align}
Y_{n,1}^{(1,1)}(x,y) = P_{n,0}^{(1,1)}(x), \qquad n \geq 0;\label{eq:m1d1Ops}
\end{align}
if $m = 1$ and $\deg \phi = 2$, then
\begin{align}
Y_{0,1}^{(1,2)}(x,y)  = P_{0,0}^{(1,2)}(x), \qquad
\left(  
\begin{array}{c}
Y^{(1,2)}_{n,1}(x,y) \\
Y^{(1,2)}_{n,2}(x,y)
\end{array}
\right)
 = 
 \left(
 \begin{array}{l}
 P_{2n-1,0}^{(1,2)}(x) \\
 P_{2n,0}^{(1,2)}(x)
 \end{array}
 \right).
\end{align}
If $m = 2$ and $\deg \phi = 1$ or $\deg \phi = 2$, then
\begin{align}
Y_{0,1}^{(2,d)}(x,y)  = P_{0,0}^{(2,d)}(x), \qquad
\left(  
\begin{array}{c}
Y^{(2,d)}_{n,1}(x,y) \\
Y^{(2,d)}_{n,2}(x,y)
\end{array}
\right)
 = 
 \left(
 \begin{array}{l}
 P_{n,0}^{(2,d)}(x) \\
 P_{n,1}^{(2,d)}(x,y)
 \end{array}
 \right), \qquad d = 1, 2. \label{eq:quadexOPs}
 \end{align}
If $m = 2$ and  $\deg \phi = 3$, then $Y_{0,1}^{(2,3)} = P_{0,0}^{(2,3)}$, $Y_{1,1}^{(2,3)} = P_{1,0}^{(2,3)}$, $Y_{1,2}^{(2,3)} = P_{1,1}^{(2,3)}$ and
\begin{align}
\left(
\begin{array}{c}
Y_{2n,1}^{(2,3)} \\
Y_{2n,2}^{(2,3)} \\
Y_{2n,3}^{(2.3)}
\end{array}
\right) = 
\left(
\begin{array}{l}
P_{3n-1,0}^{(2,3)} \\
P_{3n-1,1}^{(2,3)} \\
P_{3n,0}^{(2,3)}
\end{array}
\right), \qquad
\left(
\begin{array}{c}
Y_{2n+1,1}^{(2,3)} \\
Y_{2n+1,2}^{(2,3)} \\
Y_{2n+1,3}^{(2,3)}
\end{array}
\right) = 
\left(
\begin{array}{l}
P_{3n,1}^{(2,3)} \\
P_{3n+1,0}^{(2,3)} \\
P_{3n+1,1}^{(2,3)}
\end{array}
\right), \quad n \geq 1.  \label{eq:cubicexOPs}
\end{align}
If $m = 2$, $\deg \phi = 4$, $\phi$ is even, $\mathrm{supp}(w) = (-a, a)$, $a > 0$ and the weight $w$ is even, so that $P^{(2,4)}_{n,0}$ has the same parity as $n$, then $Y_{n,i}^{(2,4)} = P_{n,i-1}^{(2,4)}$ for $(n,i) = (0,1)$, $(1,1)$, $(1,2)$, $(2,1)$, $(2,2)$, $Y_{2,3}^{(2,4)} = P_{4,0}^{(2,4)}$ and for $n \geq 1$
\begin{align}
\left(
\begin{array}{c}
Y_{2n+1,1}^{(2,4)} \\
Y_{2n+1,2}^{(2,4)} \\
Y_{2n+1,3}^{(2,4)} \\
Y_{2n+1,4}^{(2,4)}
\end{array}
\right) = 
\left(
\begin{array}{l}
P_{4n-1,0}^{(2,4)} \\
P_{4n-1,1}^{(2,4)} \\
P_{4n+1,0}^{(2,4)} \\
P_{4n+1,1}^{(2,4)}
\end{array}
\right), \qquad 
\left(
\begin{array}{c}
Y_{2n+2,1}^{(2,4)} \\
Y_{2n+2,2}^{(2,4)} \\
Y_{2n+2,3}^{(2,4)} \\
Y_{2n+2,4}^{(2,4)}
\end{array}
\right) = 
\left(
\begin{array}{l}
P_{4n,1}^{(2,4)} \\
P_{4n+2,0}^{(2,4)} \\
P_{4n+2,1}^{(2,4)} \\
P_{4n+4,0}^{(2,4)}
\end{array}
\right).  \label{eq:evenquarOPs}
\end{align}
\end{prop}
\begin{proof}
The sets of $Y^{(m,d)}_{n,k}$ given above are orthonormal (see Proposition~\ref{prop:Pmorthog}) and the dimensions of $\CV_n$ are as in Proposition~\ref{prop:orthdimmd}.  What remains to be shown is that the $Y^{(m,d)}_{n,k}$ have degree $n$ on $\g_{m,d}$.  We only prove the latter case (\ref{eq:evenquarOPs}) since the proofs for the remaining cases are similar. 

Since $\phi$ is assumed to be even and of degree 4, $y^2 = c_4x^4 + c_2x^2 + c_0$, with $c_4 \neq 0$. To show that $Y_{2n+1,1}^{(2,4)} = P^{(2,4)}_{4n-1,0}$ has degree $2n+1$ on $\g_{2,4}$, it suffices to consider the highest degree monomial of $P^{(2,4)}_{4n-1,0}$, viz.~$x^{4n-1}$ and note that 
\begin{equation*}
x^{4n-1} = x^3\left( x^4 \right)^{n-1} = x^3\left(\frac{y^2 - c_2x^2 -c_0}{c_4}   \right)^{n-1},
\end{equation*}
which is an expression of degree $2n + 1$ whose degree cannot be further reduced on the curve $\g_{2,4}$. We conclude that $\deg_{\g_{2,4}}\left(P^{(2,4)}_{4n-1,0}\right) = 2n + 1$. Similarly, considering $Y_{2n+1,2}^{(2,4)} = P^{(2,4)}_{4n-1,1}$, the highest degree monomial is $yx^{4n-2} = yx^2\left( x^4 \right)^{n-1} = yx^2[(y^2 - c_2x^2 -c_0)/c_4]^{n-1}$ and it follows that $\deg_{\g_{2,4}}\left(P^{(2,4)}_{4n-1,1}\right) = 2n + 1$. The remaining cases  can be similarly confirmed. 

\end{proof}

The OP bases on quadratic curves (\ref{eq:quadexOPs}) and cubic curves (\ref{eq:cubicexOPs}) are the cases studied in, respectively, \cite{OX2} and \cite{FOX}.

\section{OPs on $y = \phi$, $\deg \phi \geq 3$ and $y^2 = \phi$, $\deg \phi \geq 4$}\label{sect:lanczos}
\setcounter{equation}{0}

We now consider the construction of OPs on $y = \phi$ and $y^2 = \phi$ in cases for which we are unable to give an explicit OP basis, which are the curves $y = \phi$, $\deg \phi \geq 3$ and $y^2 = \phi$, $\deg \phi \geq 4$.

By the same reasoning used in the proof of Proposition~\ref{prop:exOPs}, we deduce that
\begin{equation}
Y_{n,k}^{(1,d)} = P_{n,k-1}^{(1,d)}, \qquad (n, k) = (0, 1), (1, 1),  \label{eq:YPinm1}
\end{equation}
for $d \geq 3$ and
\begin{equation}
Y_{n,k}^{(2,d)} = P_{n,k-1}^{(2,d)}, \qquad (n,k) = (0,1), (1,1), (1,2), (2,1), (2,2), \label{eq:YPinm2}
\end{equation}
for $d \geq 4$. To generate higher degree OPs, we shall apply the Gram--Schmidt process and show that it is equivalent to the Lanczos algorithm applied to two symmetric, commuting matrices. 

We introduce the following infinite quasimatrices for the $P$-bases,
\begin{align}
\mathbf{P}^{(1,d)}=\mathbf{P}^{(1,d)}(x) &= \left(
\begin{array}{c c c c}
P^{(1,d)}_{0,0}(x) & P^{(1,d)}_{1,0}(x) & P^{(1,d)}_{2,0}(x) & \cdots
\end{array}  
\right), 
\label{eq:P1dqms} \\
\mathbf{P}^{(2,d)}=\mathbf{P}^{(2,d)}(x,y) &= \left(
\begin{array}{c | c | c | c}
\PP^{(2,d)}_{0}(x) & \PP^{(1,d)}_{1}(x,y) & \PP^{(1,d)}_{2}(x,y) & \cdots
\end{array}  
\right), 
\label{eq:P2dqms}
\end{align}
where the $\PP^{(2,d)}_n$ denote 
the block-quasimatrices 
\begin{equation}
\PP^{(2,d)}_{0}(x) =  \left( 
\begin{array}{c}
P_{0,0}^{(2,d)}(x)
\end{array}
\right), \quad
\PP^{(2,d)}_{n}(x,y) =  \left( 
\begin{array}{c c}
P_{n,0}^{(2,d)}(x) & P_{n,1}^{(2,d)}(x,y)
\end{array}
\right), \quad n \geq 1. \label{eq:P2blocks}
\end{equation}
Similarly, for the OPs,
\begin{equation}
\mathbf{Y}^{(m,d)} = \mathbf{Y}^{(m,d)}(x,y) = \left(
\begin{array}{c | c | c | c}
\YY^{(m,d)}_{0} & \YY^{(m,d)}_{1}& \YY^{(m,d)}_{2} & \cdots
\end{array}  
\right), 
\label{eq:Ymdqm}
\end{equation}
with
\begin{equation}
\YY^{(m,d)}_{n}(x,y) =  \left( 
\begin{array}{c c c}
Y_{n,1}^{(m,d)}(x,y) & \cdots & Y_{n,\dim\CV_n}^{(m,d)}(x,y)
\end{array}
\right), \qquad n \geq 0,   \label{eq:Ymdqmbs}
\end{equation}
where, since $m < d = \deg \phi$ for the cases we consider in this section, 
\begin{equation*}
\dim\CV_n = \min\{n+1, d\},
\end{equation*}
see Proposition~\ref{prop:orthdimmd}.

We can express (\ref{eq:Y1exp})--(\ref{eq:Y2exp}) as
\begin{equation}
\mathbf{Y}^{(m,d)} = \mathbf{P}^{(m,d)}C^{(m,d)},  \label{eq:YPmdconnect}
\end{equation}
where $C^{(m,d)}$ 
is referred to as the connection matrix and its entries $C^{(m,d)}_{i,j}$, $i,j \geq 0$ are the connection coefficients.  For all the values of $(m,d)$ considered in Proposition~\ref{prop:exOPs}, $C^{(m,d)} = I$, except for the final case (\ref{eq:evenquarOPs}), where there exists a permutation matrix $\mathcal{P}$ such that $C^{(2,4)}\mathcal{P} = I$. Figures~\ref{fig:C_m_1} and~\ref{fig:Cexamples_curve} depict examples of connection matrices for $m = 1, \deg \phi \geq 3$ and $m = 2, \deg \phi \geq 4$, the construction of which we shall describe in the rest of this section.

\begin{figure}[h!]
\mbox{ \includegraphics[width=0.43\textwidth]{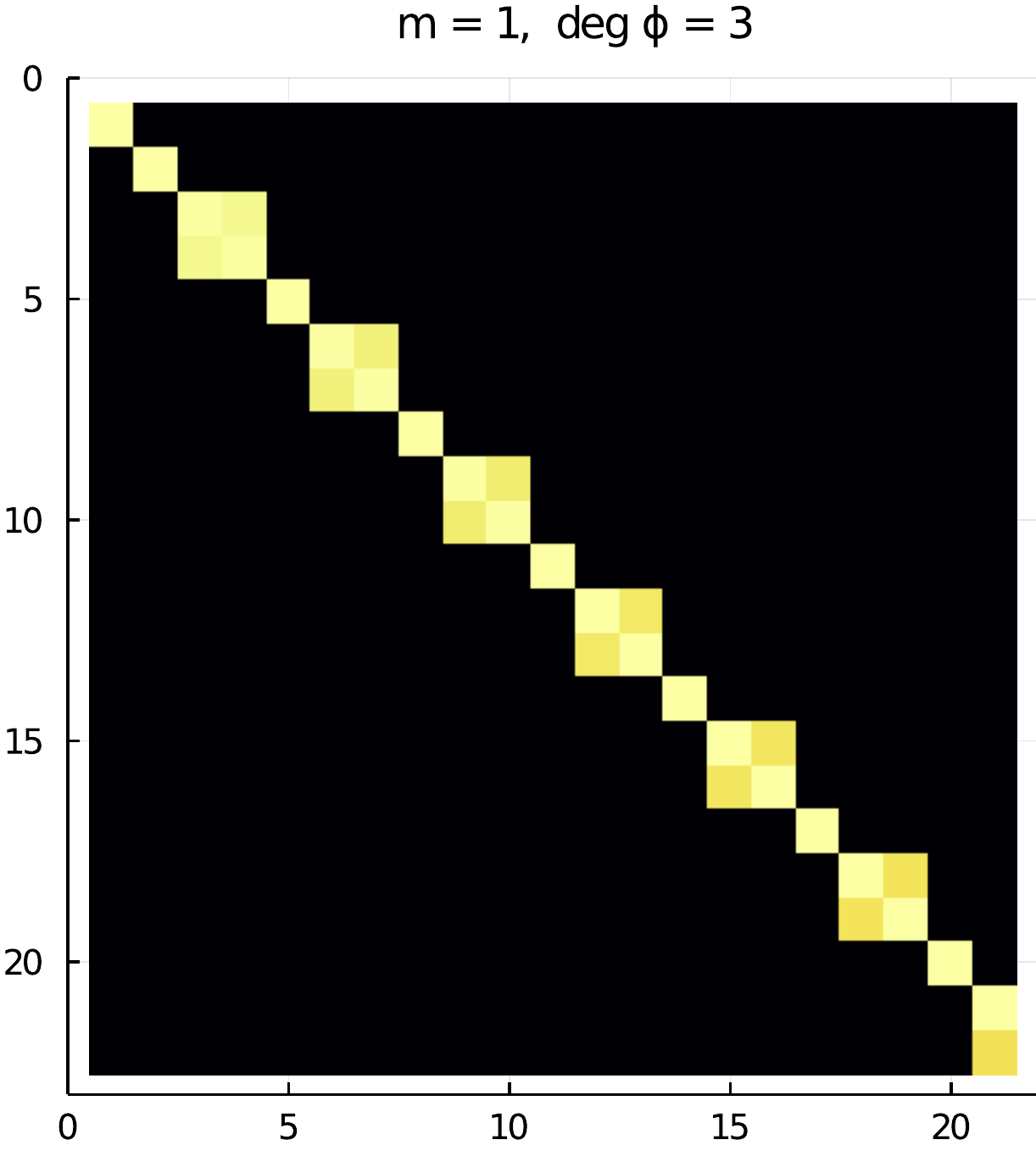} \includegraphics[width=0.49\textwidth]{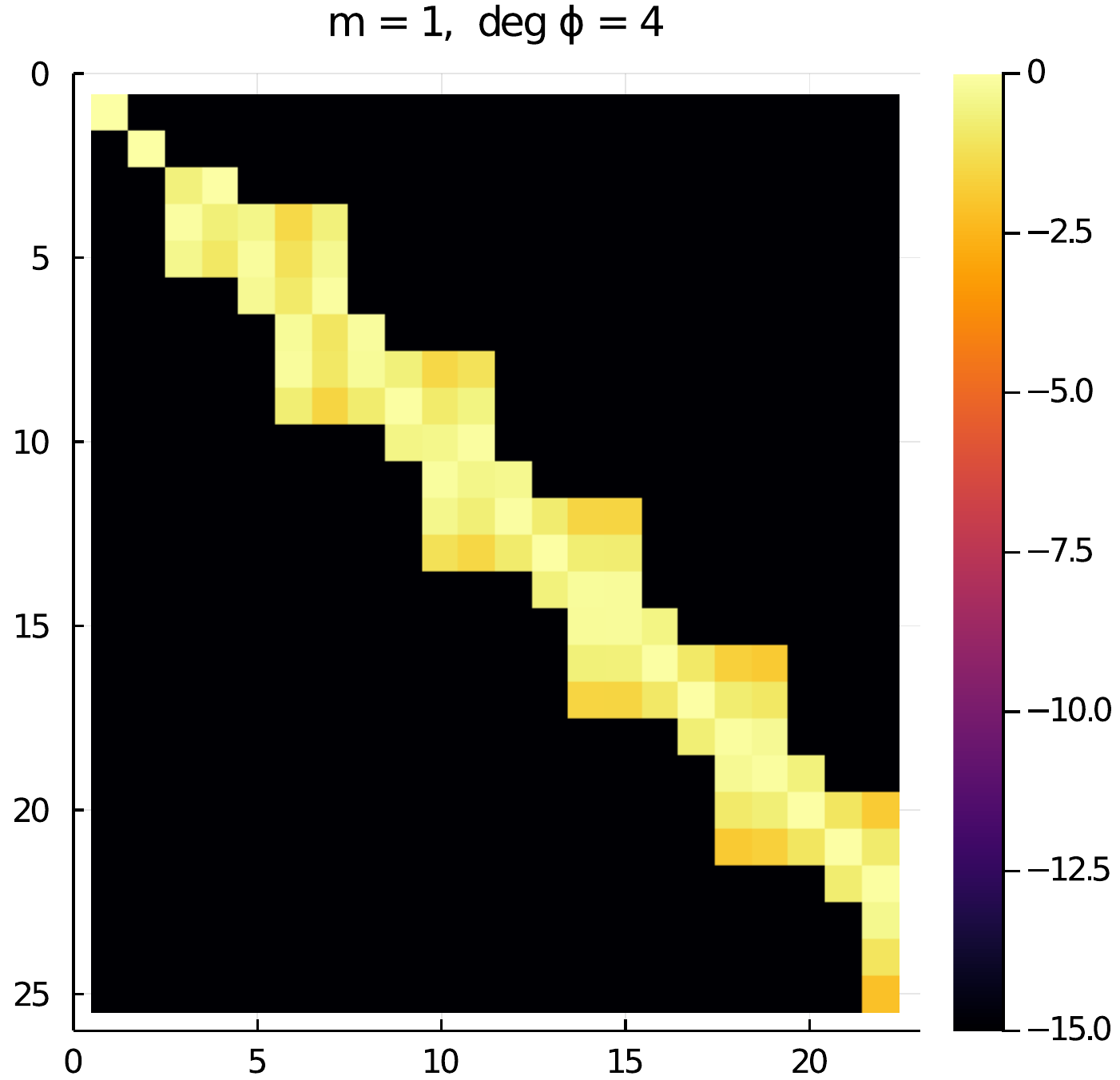} } \\
\mbox{ \includegraphics[width=0.42\textwidth]{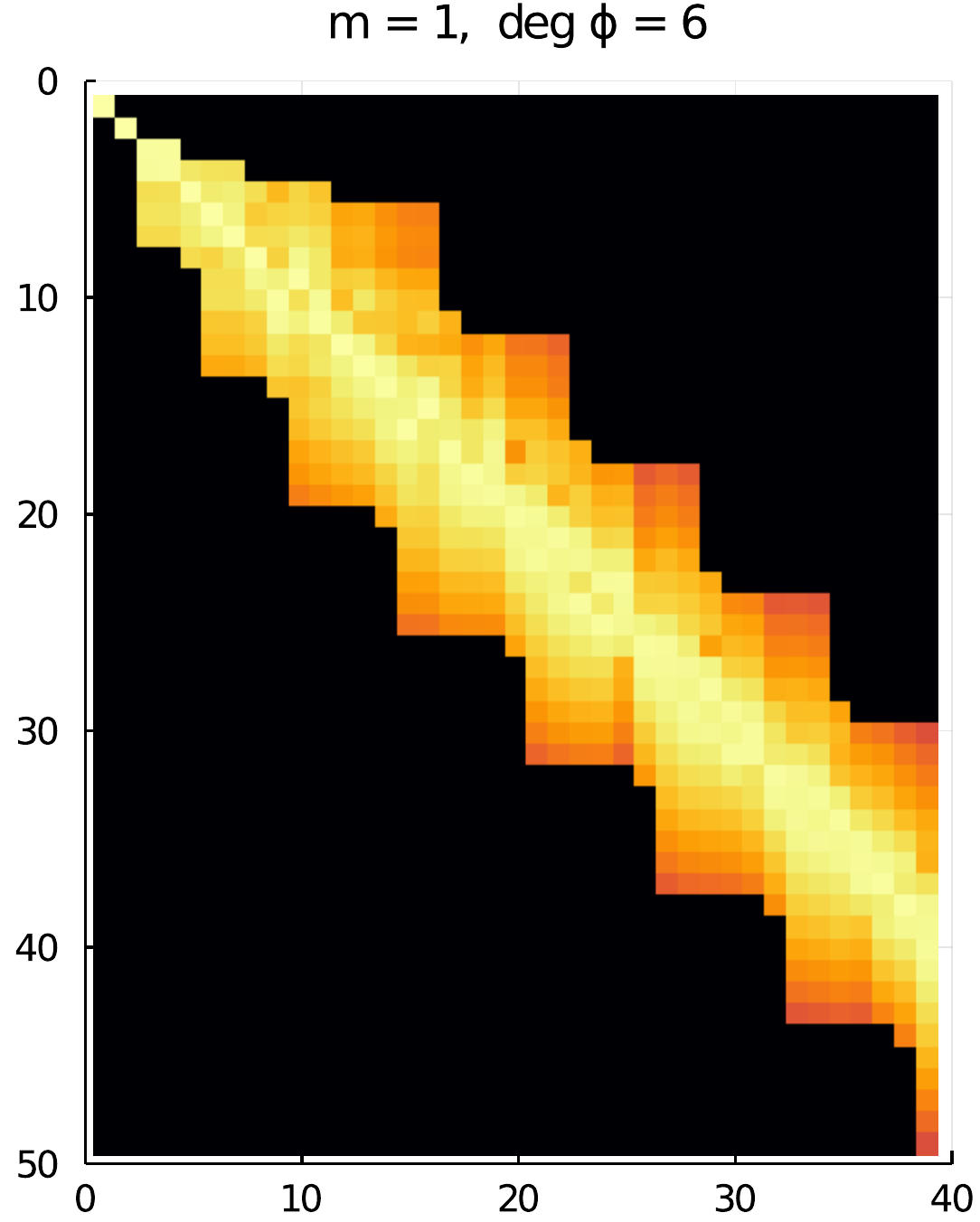} \includegraphics[width=0.5\textwidth]{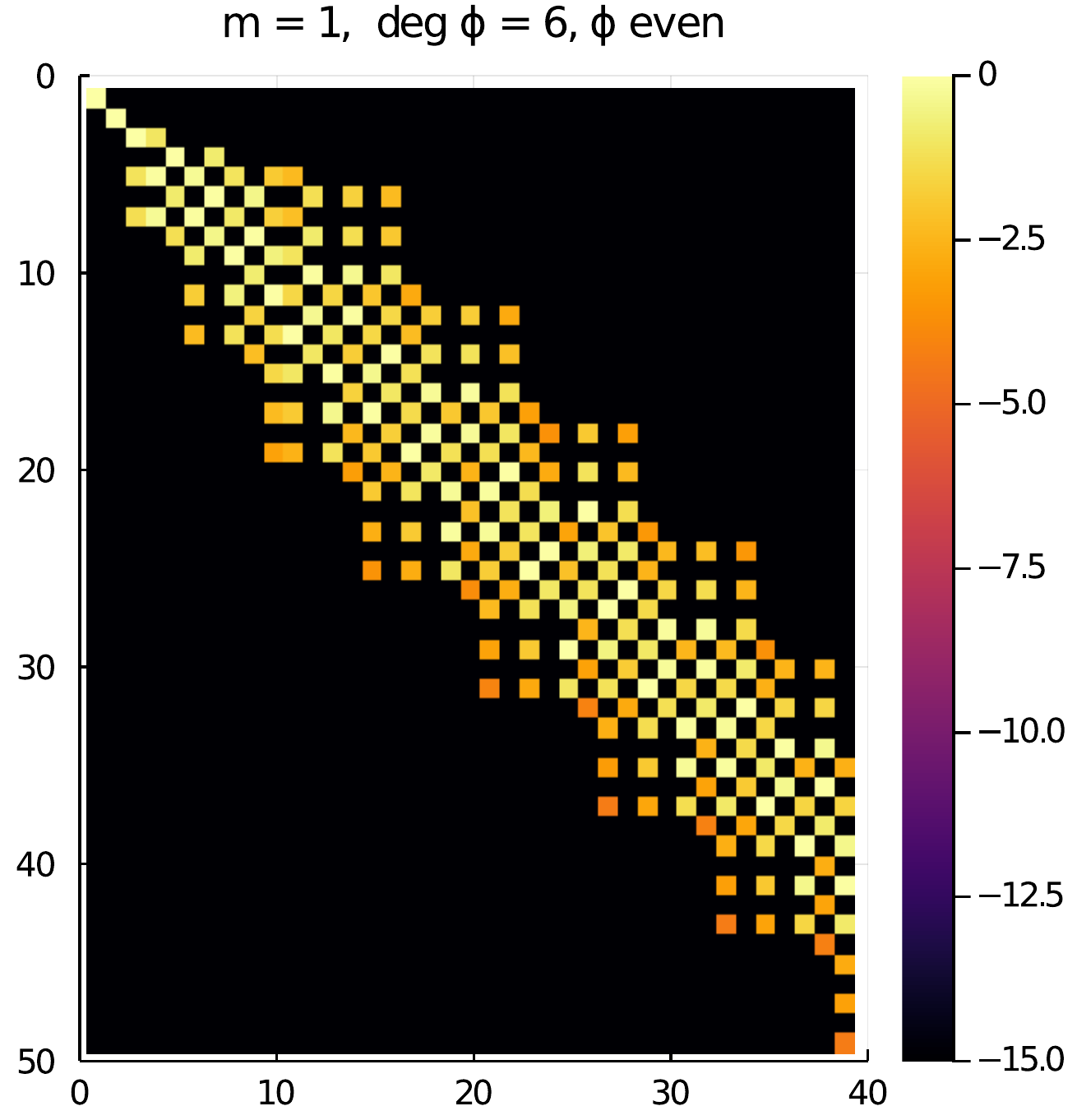} } 
  \caption{Plots of $\log_{10}$ of the entries of the connection matrices $C^{(1,d)}$ for OPs on $y = \phi(x)$. The curves are defined by $\phi(x) = x^3 - x^2 + x + 1$ (top-left), $\phi(x)  = x^4 + x^3 - x^2 + x + 1$ (top-right), $\phi(x) = (x^6 + x^5 - x^4 + x^3 - x^2 + x + 1)/10$ (bottom-left) and $\phi(x) = x^6 - x^4 - x^2  + 1$ (bottom right). The uniform weight $w(x) = 1$ on $[-1, 1]$ was used in all cases. }~\label{fig:C_m_1}
\end{figure}

\begin{figure}[h!]
\mbox{ \includegraphics[width=0.445\textwidth]{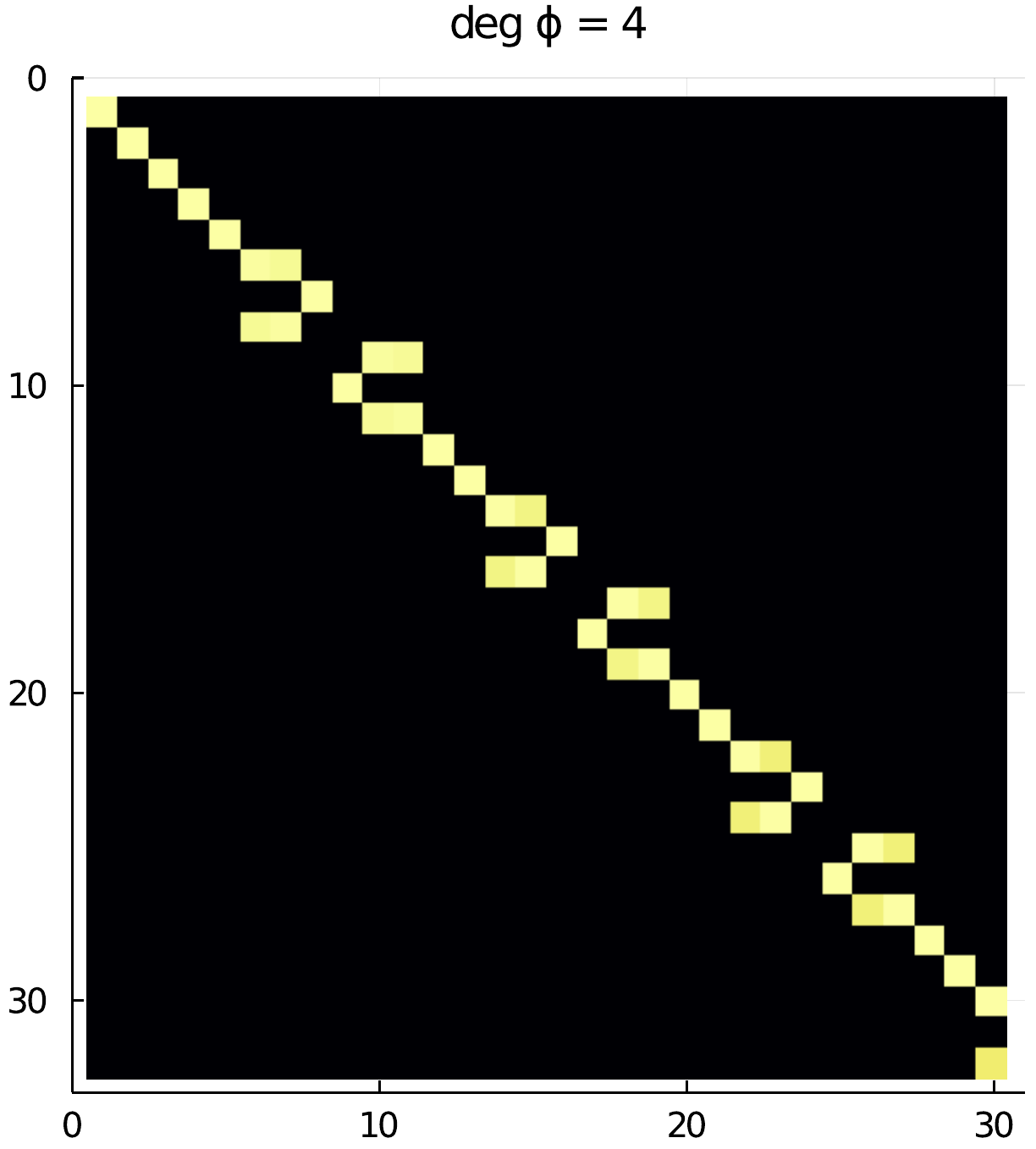} \includegraphics[width=0.465\textwidth]{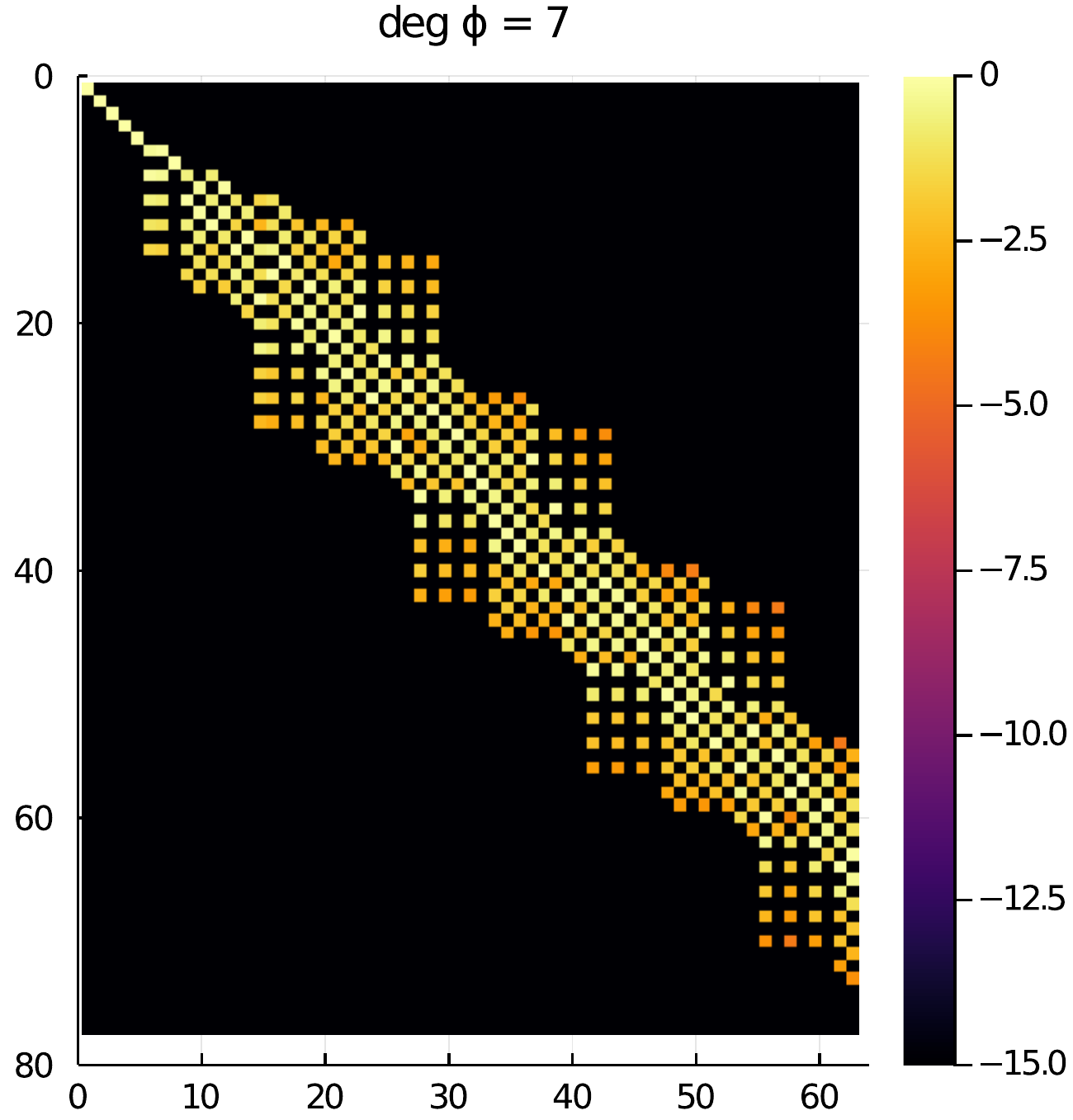} } \\
\mbox{ \includegraphics[width=0.415\textwidth]{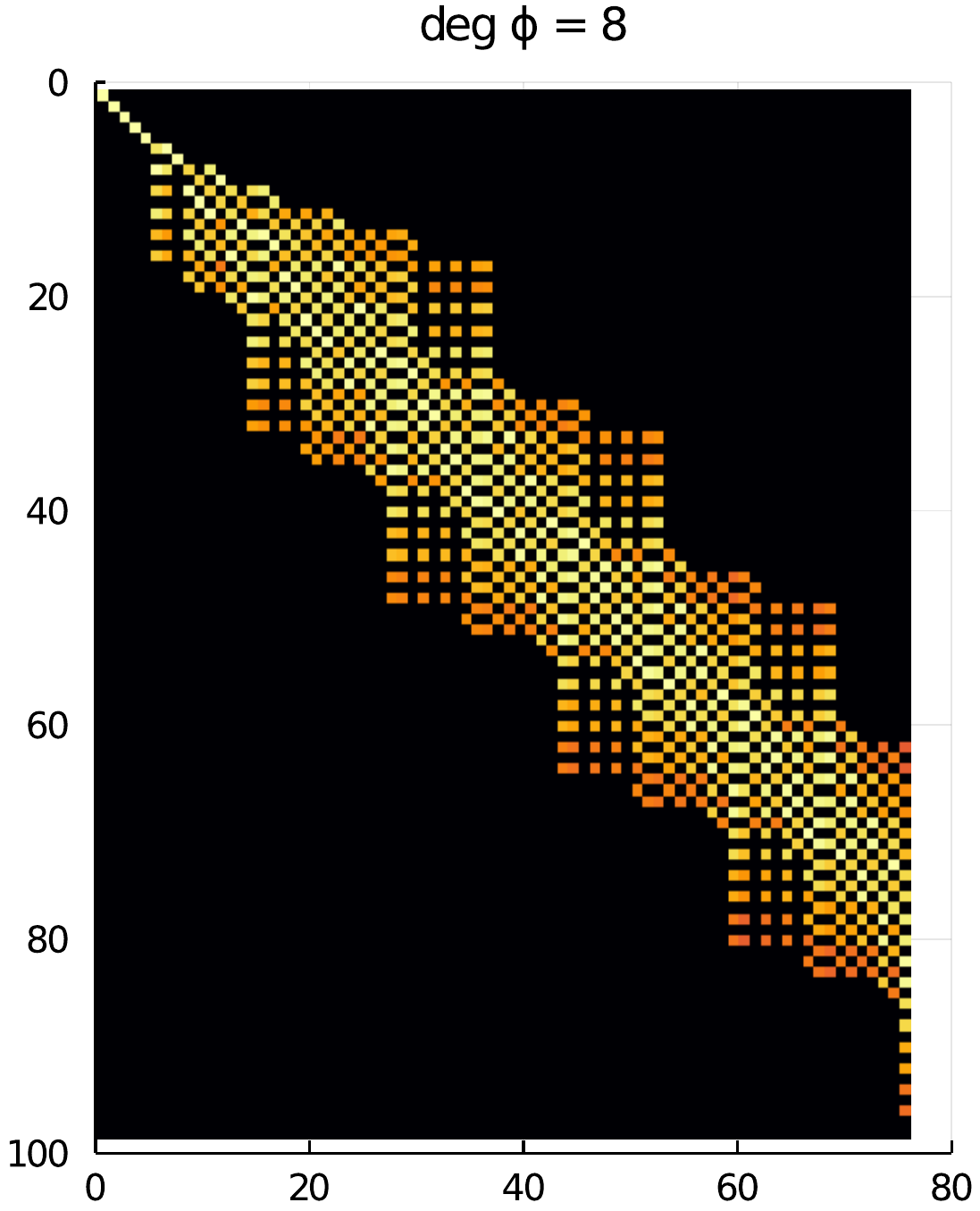} \includegraphics[width=0.495\textwidth]{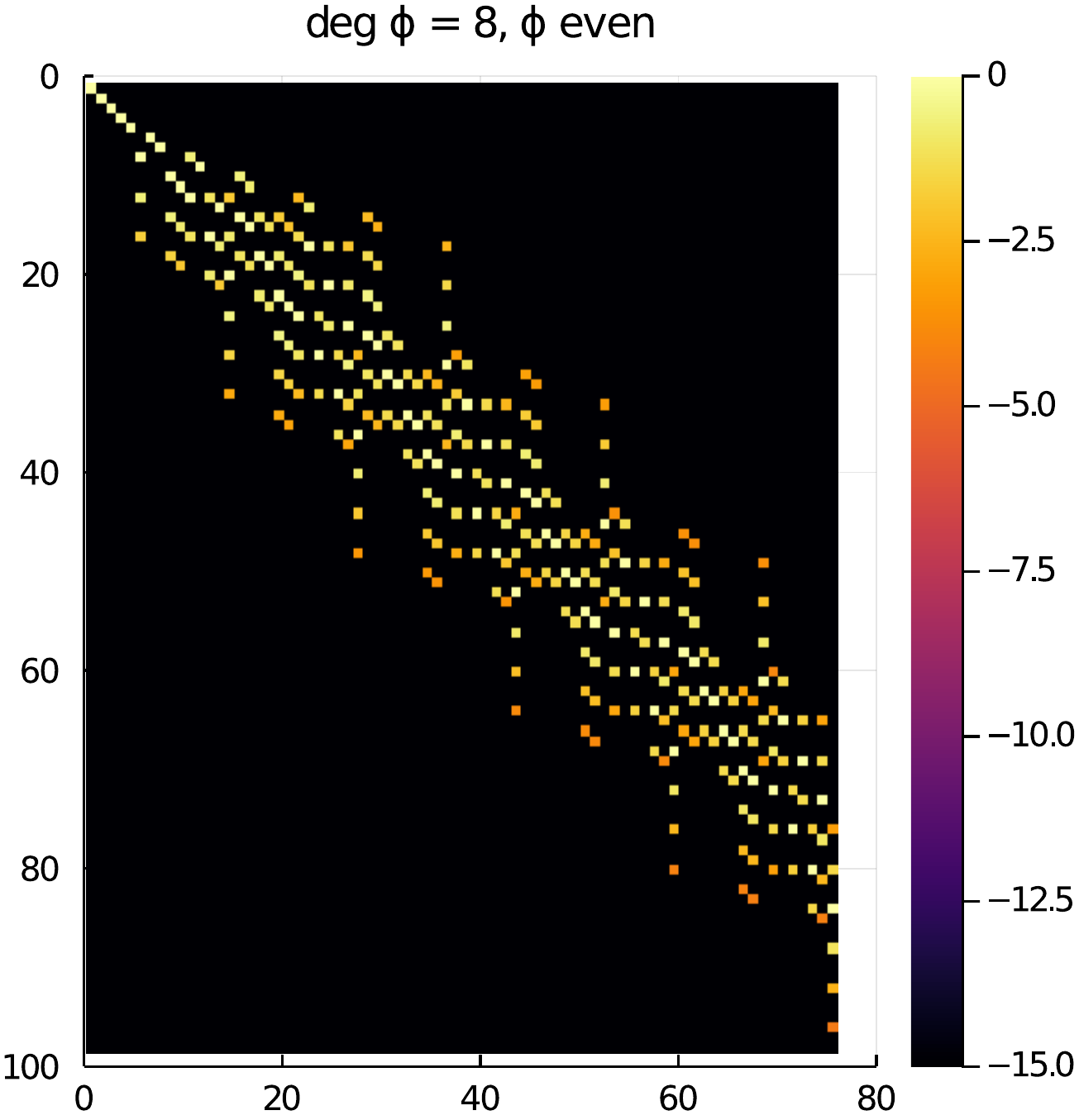} } 
  \caption{Plots of $\log_{10}$ of the entries of the connection matrices $C^{(2,d)}$ for OPs on $y^2 = \phi(x)$. 
   The curves in these examples are $\phi = (2+x)(1-x^2)(5/4-x)$ (top left), $\phi = \phi_7 := (1.15+x)(1.05+x)(1-x^2)(1.1-x)(1.2-x)(1.3-x)$ (top right), $\phi = (1.25+x)\phi_7$ (bottom left) and $\phi = (1-x^2)(1.05^2-x^2)(1.1^2-x^2)(1.15^2-x^2)$ (bottom right). The uniform weight $w = 1$ on $[-1, 1]$ was used in all cases.   }~\label{fig:Cexamples_curve}
\end{figure} 

\begin{prop}
The connection matrix $C^{(m,d)}$ for $m = 1, 2$ is orthogonal.
\end{prop}
\begin{proof}
Since the $\mathbf{Y}^{(m,d)}$ are orthonormal OPs and the $\mathbf{P}^{(m,d)}$ are orthonormal for $m = 1, 2$, it follows that $\la \mathbf{Y}^{(m,d)}, \mathbf{Y}^{(m,d)} \ra_{\g_{m,d},w} = I$ and $\la \mathbf{P}^{(m,d)}, \mathbf{P}^{(m,d)} \ra_{\g_{m,d},w} = I$, where $\langle \cdot, \cdot \rangle_{\g_{m,d},w}$ for quasimatrices is defined as in (\ref{eq:qmipd})--(\ref{eq:qmipent}), mutatis mutandis. Thus, the connection matrix is orthogonal because
\begin{align*}
I &= \la \mathbf{Y}^{(m,d)}, \mathbf{Y}^{(m,d)} \ra_{\g_{m,d},w} = \left(C^{(m,d)}\right)^{\top} \la \mathbf{P}^{(m,d)}, \mathbf{P}^{(m,d)} \ra_{\g_{m,d},w}C^{(m,d)}\\
& = \left(C^{(m,d)}\right)^{\top} C^{(m,d)}, \qquad m = 1, 2.
\end{align*} 
\end{proof}

\subsection{Multiplication matrices for the  ${P}$-bases}\label{sect:multP}
To construct the connection matrices we shall need matrices $\mathcal{X}^{(m,d)}$ and $\mathcal{Y}^{(m,d)}$ representing multiplication of $\mathbf{P}^{(m,d)}$ by $x$ and $y$:
\begin{equation}
x\mathbf{P}^{(m,d)} = \mathbf{P}^{(m,d)}\mathcal{X}^{(m,d)}, \qquad y\mathbf{P}^{(m,d)} = \mathbf{P}^{(m,d)}\mathcal{Y}^{(m,d)}.  \label{eq:xyPmd}
\end{equation} 
\begin{prop}
The multiplication matrices of the  $\mathbf{P}^{(1,d)}$ basis are
\begin{equation}
\mathcal{X}^{(1,d)} = J(w), \qquad \mathcal{Y}^{(1,d)} = \phi(J(w)),  \label{eq:xym1ops}
\end{equation}
where $J(w)$ is the Jacobi matrix of the orthonormal OPs $\{ p_n(w) \}$, which is defined as in (\ref{eq:sc3term})--(\ref{eq:jacphiw}) but with $\phi w$ replaced by $w$.  Furthermore, $\mathcal{X}^{(1,d)}$ and $\mathcal{Y}^{(1,d)}$ commute, $\mathcal{X}^{(1,d)}$ is symmetric and tridiagonal and  $\mathcal{Y}^{(1,d)}$ is symmetric with bandwidths $(d, d)$.

For $m = 2$, $\mathcal{X}^{(2,d)}$ and $\mathcal{Y}^{(2,d)}$ commute, $\mathcal{X}^{(2,d)}$ is symmetric, block-tridiagonal with diagonal blocks:
\begin{align*}
\mathcal{X}^{(2,d)} = 
\left(
\begin{array}{c c c c c}
\mathcal{X}_{0,0}^{(2,d)} & \left(\mathcal{X}_{1,0}^{(2,d)}\right)^{\top} &  & &  \\
\mathcal{X}_{1,0}^{(2,d)} & \mathcal{X}_{1,1}^{(2,d)} & \left(\mathcal{X}_{2,1}^{(2,d)}\right)^{\top} & & \\
 & \mathcal{X}_{2,1}^{(2,d)} & \mathcal{X}_{2,2}^{(2,d)} & \left(\mathcal{X}_{3,2}^{(2,d)}\right)^{\top} & \\
 & & \ddots & \ddots & \ddots
\end{array}
\right),
\end{align*}
where 
 $\mathcal{X}_{0,0}^{(2,d)} = \left(\begin{array}{c} \a_0(w) \end{array}\right) \in \RR^{1\times 1} $, $\mathcal{X}_{1,0}^{(2,d)} = \left(\begin{array}{c c} \b_0(w) & 0 \end{array}\right)^{\top} \in \RR^{2\times 1}$, and for $n \geq 1$
\begin{align*}
\mathcal{X}_{n,n}^{(2,d)} = \left(
\begin{array}{c c}
\a_n(w) & 0 \\
0 & \a_{n-1}(\phi w)
\end{array}
\right), \qquad
\mathcal{X}_{n+1,n}^{(2,d)} = \left(
\begin{array}{c c}
\b_n(w) & 0 \\
0 & \b_{n-1}(\phi w)
\end{array}
\right). 
\end{align*}
The matrix $\mathcal{Y}^{(2,d)}$ is symmetric with block-bandwiths $(d-1, d-1)$:
\begin{align}
\mathcal{Y}^{(2,d)} = \left(
\begin{array}{ c  c  c  c  c  c  c  c c}
0                  & \left(\mathcal{Y}_{1,0}^{(2,d)}\right)^{\top} &        &       &        &     &  & &\\
\mathcal{Y}_{1,0}^{(2,d)}& \mathcal{Y}_{1,1}^{(2,d)} & \left(\mathcal{Y}_{2,1}^{(2,d)}\right)^{\top}&\cdots & \left(\mathcal{Y}_{d,1}^{(2,d)}\right)^{\top}&    & & & \\
       & \mathcal{Y}_{2,1}^{(2,d)} & \mathcal{Y}_{2,2}^{(2,d)} & \cdots &\cdots & \left(\mathcal{Y}_{d+1,2}^{(2,d)}\right)^{\top} & & & \\
       & \vdots             & \ddots  & \ddots &  \ddots & \ddots &\ddots & & \\
                   &\mathcal{Y}_{d,1}^{(2,d)} & \cdots & \cdots & \mathcal{Y}_{d,d}^{(2,d)} &\left(\mathcal{Y}_{d+1,d}^{(2,d)}\right)^{\top} & \cdots & \left(\mathcal{Y}_{2d-1,d}^{(2,d)}\right)^{\top} & \\
                   &        & \ddots & \ddots  & \ddots & \ddots                 & \ddots & \ddots & \ddots
\end{array}
\right),  \label{eq:Y2dcop}
\end{align}
where $\mathcal{Y}_{1,0}^{(2,d)} = \left(\begin{array}{c c} 0 & r_{0,0}  \end{array} \right)^{\top}$,
\begin{align*}
\mathcal{Y}^{(2,d)}_{n,n} = \left(
\begin{array}{c c}
0 & r_{n-1,n}  \\
r_{n-1,n} & 0
\end{array}
\right), \quad
\mathcal{Y}^{(2,d)}_{n+1,n} = \left(
\begin{array}{c c}
0 & r_{n-1,n+1}  \\
r_{n,n} & 0
\end{array}
\right), \quad n \geq 1,
\end{align*}
\begin{align*}
\mathcal{Y}^{(2,d)}_{n+k,n} = \left(
\begin{array}{c c}
0 & r_{n-1,n+k}  \\
0 & 0
\end{array}
\right), \qquad
2 \leq k \leq d-1, \qquad n \geq 1,
\end{align*}
and the $r_{i,j}$ are defined in (\ref{eq:rkndef}).
\end{prop}
\begin{proof}
Since $x^k\mathbf{P}(w) = \mathbf{P}(w)\left(J(w)\right)^k$, $k \geq 0$ and $\left(J(w)\right)^k$ is symmetric with bandwidths $(k, k)$, it follows that $y\mathbf{P}(w) = \phi(x)\mathbf{P}(w) =\mathbf{P}(w) \phi(J(w))$, where $\phi(J(w))$ is symmetric with bandwidths $(d, d)$ because $\phi$ is a polynomial of degree $d$. The result  (\ref{eq:xym1ops}) then follows since $\mathbf{P}^{(1,d)} = \mathbf{P}(w)$.

The commutativity of $\mathcal{X}^{(m,d)}$ and $\mathcal{Y}^{(m,d)}$ for $m = 1, 2$ follows immediately from their definitions in (\ref{eq:xyPmd}) and the commutativity of multiplication by $x$ and $y$, $xy\mathbf{P}^{(m,d)} = yx\mathbf{P}^{(m,d)}$.

The entries of $\mathcal{X}^{(2,d)}$ follow straightforwardly from (\ref{eq:Pmddef}) and the three-term recurrences of the OP families $\{ p_n(w) \}$ and $\{ p_n(\phi w) \}$.

To construct $\mathcal{Y}^{(2,d)}$, we first note from (\ref{eq:Pmddef}) and (\ref{eq:raiseop})--(\ref{eq:raisestruct}) that
\begin{align}
yP_{n,0}^{(2,d)} & = r_{n-d,n} P_{n-d+1,1}^{(2,d)} + \cdots + r_{n,n} P_{n+1,1}^{(2,d)}, \label{eq:ymultyp}  
\end{align}
for $n \geq 0$, where $r_{i,j} = 0$ if $i < 0$. Since $yP_{n+1,1}^{(2,d)} = y^2p_{n}(\phi w)/\sqrt{2} = \phi p_{n}(\phi w)/\sqrt{2}$, we can express  $yP_{n+1,1}^{(2,d)}$ in the $\mathbf{P}^{(2,d)}$ basis using the entries of the lowering matrix $L$ defined by
\begin{equation*}
\phi(x) \mathbf{P}(\phi w) = \mathbf{P}(w)L,
\end{equation*}
where (cf.~(\ref{eq:qmipd}))
\begin{align*}
L  = \langle \mathbf{P}(w),\phi \mathbf{P}(\phi w) \rangle_{w} = \int \mathbf{P}^{\top}(w) \mathbf{P}(\phi w) \phi w \d x = R^{\top},
\end{align*}
and therefore
\begin{align}
 yP_{n+1,1}^{(2,d)} & = r_{n,n}P_{n,0}^{(2,d)} + \cdots + r_{n,n+d} P_{n+d,0}^{(2,d)}. \label{eq:ymultyp1} 
\end{align}
The entries and structure of $\mathcal{Y}^{(2,d)}$ then follow from (\ref{eq:ymultyp}) and (\ref{eq:ymultyp1}).
\end{proof}

\subsection{The Gram--Schmidt process viewed as a block Lanczos algorithm} 

Since the $Y_{n,k}^{(m,d)}$ are orthonormal OPs, the $\YY_n^{(m,d)}$ satisfy three-term recurrences of the form~\cite{DX}
\begin{align}
x\YY_n^{(m,d)} = \YY_{n-1}^{(m,d)} \left(B_{n-1,x}^{(m,d)}\right)^{\top} +  \YY_n^{(m,d)} A_{n,x}^{(m,d)} + \YY_{n+1}^{(m,d)} B_{n,x}^{(m,d)},  \label{eq:xYmdrec}
\end{align}
and
\begin{align}
y\YY_n^{(m,d)} = \YY_{n-1}^{(m,d)} \left(B_{n-1,y}^{(m,d)}\right)^{\top} +  \YY_n^{(m,d)} A_{n,y}^{(m,d)} + \YY_{n+1}^{(m,d)} B_{n,y}^{(m,d)},  \label{eq:yYmdrec}
\end{align}
for $n \geq 0$ with $\YY_{-1}:=0$, where
\begin{align}
A_{n,x}^{(m,d)} = \la \YY_n^{(m,d)}, x\YY_{n}^{(m,d)} \ra_{\g_{m,d},w}, \quad 
B_{n,x}^{(m,d)} = \la \YY_{n+1}^{(m,d)}, x\YY_{n}^{(m,d)} \ra_{\g_{m,d},w}, \label{eq:ABxmd}
\end{align}
and $A_{n,y}^{(m,d)}$, $B_{n,y}^{(m,d)}$  are defined similarly. Hence,  $A_{n,x}^{(m,d)}, A_{n,y}^{(m,d)} \in \RR^{\dim \CV_n \times \dim \CV_n}$, $B_{n,x}^{(m,d)}, B_{n,y}^{(m,d)} \in \RR^{\dim \CV_{n+1} \times \dim \CV_n}$ and $A_{n,x}^{(m,d)}$ and  $A_{n,y}^{(m,d)}$ are symmetric.
In quasimatrix notation, the recurrences (\ref{eq:xYmdrec}) and (\ref{eq:yYmdrec}) can be expressed as
\begin{equation*}
x\mathbf{Y}^{(m,d)} = \mathbf{Y}^{(m,d)}J_x^{(m,d)}, \qquad y\mathbf{Y}^{(m,d)} = \mathbf{Y}^{(m,d)}J_y^{(m,d)},
\end{equation*}
where $J_x^{(m,d)}$ and $J_y^{(m,d)}$ are the symmetric, commuting, block-tridiagonal Jacobi matrices of the OPs $\mathbf{Y}^{(m,d)}$.

To express the recurrences (\ref{eq:xYmdrec}) and (\ref{eq:yYmdrec}) in the $ \mathbf{P}^{(m,d)}$ bases, we write the  connection matrix in block form
\begin{equation}
C^{(m,d)} = \left( 
\begin{array}{c | c | c | c}
\CC_0^{(m,d)} & \CC_1^{(m,d)} & \CC_2^{(m,d)} & \cdots
\end{array}\right),
\label{eq:Cmdbs}
\end{equation}
where $\CC_n^{(m,d)} \in \mathbb{R}^{\infty \times \dim \CV_n}$,  then
\begin{equation}
\YY_n^{(m,d)} = \mathbf{P}^{(m,d)}\CC_n^{(m,d)}.  \label{eq:YPmdblock} 
\end{equation} 
 
\begin{prop}
The Jacobi matrices of the OPs $\mathbf{Y}^{(m,d)}$ and the multiplication matrices of the $\mathbf{P}^{(m,d)}$ bases are related as follows
\begin{equation}
J_x^{(m,d)} = \left(C^{(m,d)}\right)^{\top}\mathcal{X}^{(m,d)} C^{(m,d)}, \qquad J_y^{(m,d)} = \left(C^{(m,d)}\right)^{\top}\mathcal{Y}^{(m,d)} C^{(m,d)}, \quad m = 1, 2.\label{eq:jacmult}
\end{equation}
Furthermore, the $\CC_n^{(m,d)}$ satisfy the recurrence relations
\begin{align}
\mathcal{X}^{(m,d)}\CC_n^{(m,d)} = \CC_{n-1}^{(m,d)} \left(B_{n-1,x}^{(m,d)}\right)^{\top} + \CC_n^{(m,d)} A_{n,x}^{(m,d)} + \CC_{n+1}^{(m,d)} B_{n,x}^{(m,d)},  \label{eq:XCmdblock}
\end{align}
and
\begin{align}
\mathcal{Y}^{(m,d)}\CC_n^{(m,d)} = \CC_{n-1}^{(m,d)} \left(B_{n-1,y}^{(m,d)}\right)^{\top} + \CC_n^{(m,d)} A_{n,y}^{(m,d)} + \CC_{n+1}^{(m,d)} B_{n,y}^{(m,d)}. \label{eq:YCmdblock}  
\end{align}
\end{prop}
\begin{proof}
Since $x\mathbf{Y}^{(m,d)} = x\mathbf{P}^{(m,d)}C^{(m,d)}=\mathbf{P}^{(m,d)}\mathcal{X}^{(m,d)} C^{(m,d)}$,  we have that
\begin{align*}
& J_x^{(m,d)} = \la \mathbf{Y}^{(m,d)}, x\mathbf{Y}^{(m,d)} \ra_{\g_{m,d},w} = \\
&   \left(C^{(m,d)}\right)^{\top}  \la \mathbf{P}^{(m,d)}, \mathbf{P}^{(m,d)} \ra_{\g_{m,d},w} \mathcal{X}^{(m,d)} C^{(m,d)} =  \left(C^{(m,d)}\right)^{\top}\mathcal{X}^{(m,d)} C^{(m,d)},
\end{align*}
for $m = 1, 2$.

Using (\ref{eq:YPmdblock}) and setting $x\YY_n^{(m,d)} = \mathbf{P}^{(m,d)}\mathcal{X}^{(m,d)}\CC_n^{(m,d)}$ in (\ref{eq:xYmdrec}), we arrive at (\ref{eq:XCmdblock}) and (\ref{eq:YCmdblock}) can be derived similarly. 
\end{proof}


The standard Stieltjes method for computing OPs and their recurrence coefficients uses a discretised quadrature rule to compute the required inner products, see~\cite{bru,Gautschi}. Hence, the (discretised) OPs need to be evaluated at a discrete set of points. By contrast, we compute OPs in ``coefficient space'' (as opposed to ``function-value space'', as is done in the standard approach) via  (\ref{eq:XCmdblock}) and (\ref{eq:YCmdblock}) to generate subsequent blocks of the connection matrices. Computation of the inner products (\ref{eq:ABxmd}) is accomplished via
\begin{align}
B_{n-1,x}^{(m,d)} = 
\left( \CC_n^{(m,d)} \right)^{\top}\mathcal{X}^{(m,d)}\CC_{n-1}^{(m,d)}, \qquad
A_{n,x}^{(m,d)} = 
\left( \CC_n^{(m,d)} \right)^{\top}\mathcal{X}^{(m,d)}\CC_{n}^{(m,d)}, \label{eq:ABCmd}
\end{align}
which follows from (\ref{eq:ABxmd}) and (\ref{eq:YPmdblock}). The inner products required for $B_{n-1,y}^{(m,d)}$ and $A_{n,y}^{(m,d)}$ are computed similarly.  Hence, once we have the multiplication matrices $\mathcal{X}^{(m,d)}$ and $\mathcal{Y}^{(m,d)}$ (which are built from the Jacobi matrices $J(w)$, $J(\phi w)$ and the raising matrix $R$), we compute inner products via only finite matrix multiplications, with no need for quadrature of the evaluation of OPs. 




Since $J_x^{(m,d)}$ and $J_y^{(m,d)}$ are symmetric, block-tridiagonal matrices, we note from (\ref{eq:jacmult}) that the connection matrices simultaneously block-tridiagonalise the symmetric, commuting matrices $\mathcal{X}^{(m,d)}$ and $\mathcal{Y}^{(m,d)}$ via orthogonal similarity transformations. Hence, our approach, which is essentially the Gram--Schmidt process applied in coefficient space to OPs in the $\mathbf{P}^{(m,d)}$ bases via the connection coefficients, can be viewed as the Lanczos procedure applied to $\mathcal{X}^{(m,d)}$ and $\mathcal{Y}^{(m,d)}$.

\subsection{Construction of the initial columns of the connection matrices}\label{sect:Cinit}

Let $\bm{e}_k$ be  defined as in (\ref{eq:basisvec})
and let $\mathbf{C}^{(m,d)}_j$ denote column $j$ of the connection matrix $C^{(m,d)}$, with $j \geq 0$. 
\begin{prop}\label{prop:Cinit}
The first few columns of the connection matrices are given by
\begin{equation}
\CC_0^{(1,d)} = \mathbf{C}^{(1,d)}_0 = \bm{e}_0, \qquad \CC_1^{(1,d)} = \left(
\begin{array}{c c}
\bm{e}_1 & \mathbf{C}^{(1,d)}_2
\end{array}
\right),   \label{eq:C1dinblocks}
\end{equation}
and for $m = 2$,
\begin{equation}
\CC_0^{(2,d)} = \bm{e}_0, \qquad \CC_1^{(2,d)} = \left(
\begin{array}{c c}
\bm{e}_1 & \bm{e}_2
\end{array}
\right), \qquad \CC_2^{(2,d)} =  \left(
\begin{array}{c c c}
\bm{e}_3 & \bm{e}_4 & \mathbf{C}^{(2,d)}_5
\end{array}
\right),  \label{eq:C2dindblocks}
\end{equation}
where
\begin{equation}
\mathbf{C}^{(1,d)}_2 = \frac{1}{b_{2,1}^{0,y}}\left( \mathcal{Y}_{2,0}^{(1,d)}\bm{e}_2 + \cdots +  \mathcal{Y}_{d,0}^{(1,d)} \bm{e}_d  \right), \quad \left( b_{2,1}^{0,y}\right)^2 = \left( \mathcal{Y}_{2,0}^{(1,d)}\right)^2 + \cdots + \left( \mathcal{Y}_{d,0}^{(1,d)} \right)^2,  \label{eq:c12}
\end{equation}
\begin{equation}
\mathbf{C}^{(2,d)}_5 = \frac{1}{b_{3,2}^{1,y}}\left( r_{0,3}\bm{e}_5 + \cdots +  r_{0,d} \bm{e}_{2d-1}  \right), \quad \left( b_{3,2}^{1,y}\right)^2 = r_{0,3}^2 + \cdots + r_{0,d}^2,  \label{eq:c25}
\end{equation}
and $\mathcal{Y}_{i,j}^{(1,d)}$, $i, j \geq 0$ denotes the $(i,j)$ entry of $\mathcal{Y}^{(1,d)}$ (defined in (\ref{eq:xym1ops})).
Furthermore, the first few block-entries of the Jacobi matrices $J_x^{(m,d)}$ and $J_y^{(m,d)}$ are given by (see (\ref{eq:xYmdrec})--(\ref{eq:yYmdrec})) 
\begin{equation}
A_{0,x}^{(1,d)} = \left(
\begin{array}{c}
\mathcal{X}_{0,0}^{(1,d)}
\end{array}
\right), \qquad
B_{0,x}^{(1,d)} = \left(
\begin{array}{c c}
\mathcal{X}_{1,0}^{(1,d)} & 0
\end{array}
\right)^{\top}, \label{eq:ABxm1n0} 
\end{equation}
\begin{equation}
A_{0,y}^{(1,d)} = \left(
\begin{array}{c}
\mathcal{Y}_{0,0}^{(1,d)}
\end{array}
\right), \qquad
B_{0,y}^{(1,d)} = \left(
\begin{array}{c c}
\mathcal{Y}_{1,0}^{(1,d)} & b_{2,1}^{0,y}
\end{array}
\right)^{\top},  \label{eq:ABym1n0}
\end{equation}
and for $m = 2$,
\begin{equation}
A_{0,x}^{(2,d)} = \mathcal{X}^{(2,d)}_{0,0}, \quad 
B_{0,x}^{(2,d)} = \mathcal{X}^{(2,d)}_{1,0}, \quad
A_{1,x}^{(2,d)} = \mathcal{X}^{(2,d)}_{1,1}, \quad
B_{1,x}^{(2,d)} = \left(
\mathcal{X}^{(2,d)}_{2,1}
\begin{array}{c}
0 \\
0
\end{array}
\right)^{\top},\label{eq:ABxm2n0} 
\end{equation}
and
\begin{equation}
A_{0,y}^{(2,d)} = \left( \begin{array}{c}
0
\end{array} \right), \quad 
B_{0,y}^{(2,d)} = \mathcal{Y}^{(2,d)}_{1,0}, \quad
A_{1,y}^{(2,d)} = \mathcal{Y}^{(2,d)}_{1,1}, \quad
B_{1,y}^{(2,d)} = \left(
\mathcal{Y}^{(2,d)}_{2,1}
\begin{array}{c}
0 \\
b_{3,2}^{1,y}
\end{array}
\right)^{\top}.  \label{eq:ABym2n0}
\end{equation}
\end{prop}
\begin{proof}
The first few columns of $C^{(m,d)}$ as well as the matrices (\ref{eq:ABxm1n0})--(\ref{eq:ABym2n0}) follow from (\ref{eq:YPinm1}), (\ref{eq:YPinm2}) and (\ref{eq:ABxmd}). We shall obtain $\mathbf{C}^{(1,d)}_2$, $\mathbf{C}^{(2,d)}_5$ and the entries  $b_{2,1}^{0,y}$ and $b_{3,2}^{1,y}$ by constructing, respectively, $Y_{1,2}^{(1,d)}$ and $Y_{2,3}^{(2,d)}$. 

The recurrence (\ref{eq:yYmdrec}) and (\ref{eq:ABym1n0}) imply
\begin{equation*}
yY_{0,1}^{(1,d)} = \mathcal{Y}_{0,0}^{(1,d)} Y_{0,1}^{(1,d)} + \mathcal{Y}_{1,0}^{(1,d)} Y_{1,1}^{(1,d)} +  b_{2,1}^{0,y}Y_{1,2}^{(1,d)},
\end{equation*}
on the other hand, using the multiplication matrix $\mathcal{Y}^{(1,d)}$ in (\ref{eq:xym1ops}),
\begin{align*}
yY_{0,1}^{(1,d)} & = \phi(x)P_{0,1}^{(1,d)} =  \mathcal{Y}^{(1,d)}_{0,0}P_{0,0}^{(1,d)} + \cdots + \mathcal{Y}^{(1,d)}_{d,0}P_{d,0}^{(1,d)} \\
& = \mathcal{Y}_{0,0}^{(1,d)} Y_{0,1}^{(1,d)} + \mathcal{Y}_{1,0}^{(1,d)} Y_{1,1}^{(1,d)} + \mathcal{Y}_{2,0}^{(1,d)} P_{2,0}^{(1,d)} + \cdots +  \mathcal{Y}_{d,0}^{(1,d)} P_{d,0}^{(1,d)},
\end{align*}
comparing these two expressions, we have that
\begin{align*}
Y_{1,2}^{(1,d)} = \frac{1}{b_{2,1}^{0,y}}\left( \mathcal{Y}_{2,0}^{(1,d)} P_{2,0}^{(1,d)} + \cdots +  \mathcal{Y}_{d,0}^{(1,d)} P_{d,0}^{(1,d)}  \right) = \mathbf{P}^{(1,d)}\mathbf{C}^{(1,d)}_2, 
\end{align*}
from which (\ref{eq:c12}) follows and the expression for $b_{2,1}^{0,y}$ follows from the requirement that $\| Y_{1,2}^{(1,d)} \|^2_{\g_{1,d},w} = \la Y_{1,2}^{(1,d)}, Y_{1,2}^{(1,d)} \ra_{\g_{1,d},w} = 1$, which implies that
\begin{equation*}
\la Y_{1,2}^{(1,d)}, Y_{1,2}^{(1,d)} \ra_{\g_{1,d},w} =  \la \mathbf{P}^{(1,d)}\mathbf{C}^{(1,d)}_2, \mathbf{P}^{(1,d)}\mathbf{C}^{(1,d)}_2 \ra_{\g_{1,d},w} = \left(\mathbf{C}^{(1,d)}_2\right)^{\top}\mathbf{C}^{(1,d)}_2 = 1;
\end{equation*}
  $\mathbf{C}^{(1,d)}_5$ and $b_{3,2}^{1,y}$ can be derived similarly  by using (\ref{eq:ABym2n0}) and (\ref{eq:ymultyp1}). 
\end{proof}
%
%
%
%
%
Note from (\ref{eq:c12})--(\ref{eq:c25}) that the entries $b^{0,y}_{2,1}$ and $b^{1,y}_{3,2}$, whose values are determined by the normalisation of the OPs, are unique up to a choice of sign. Throughout, we shall let entries that arise from normalisation be positive.

\subsection{Sequences of orthogonalisation for the subsequent columns of the connection matrices}

We now consider the sequence of multiplications by $x$ and $y$ to be used for constructing the OPs. We first note that different sequences of multiplication can result in different Jacobi matrices $J_x^{(m,d)}$, $J_y^{(m,d)}$ and connection matrices $C^{(m,d)}$. For example, suppose $m = 1$, we set $Y^{(1,d)}_{0,1} = P_{0,0}^{(1,d)}$, first multiply by $y$ and orthogonalise to obtain $Y^{(1,d)}_{1,1}$, and then multiply by $x$ and orthogonalise to obtain $Y^{(1,d)}_{1,2}$, then, cf.~(\ref{eq:ABxm1n0}) and (\ref{eq:ABym1n0}),
\begin{equation*}
A_{0,y}^{(1,d)} = \left(
\begin{array}{c}
\mathcal{Y}_{0,0}^{(1,d)}
\end{array}
\right), \qquad
B_{0,y}^{(1,d)} = \left(
\begin{array}{c c}
b_{1,1}^{0,y} & 0
\end{array}
\right)^{\top}, 
\end{equation*}
\begin{equation*}
A_{0,x}^{(1,d)} = \left(
\begin{array}{c}
\mathcal{X}_{0,0}^{(1,d)}
\end{array}
\right), \qquad
B_{0,x}^{(1,d)} = \left(
\begin{array}{c c}
b^{0,x}_{1,1} & b^{0,x}_{2,1}
\end{array}
\right)^{\top}, 
\end{equation*}
and
\begin{align*}
Y_{1,1}^{(1,d)} = \frac{1}{b_{1,1}^{0,y}}\left( \mathcal{Y}_{1,0}^{(1,d)} P_{1,0}^{(1,d)} + \cdots +  \mathcal{Y}_{d,0}^{(1,d)} P_{d,0}^{(1,d)}  \right) = \mathbf{P}^{(1,d)} \mathbf{C}_{1}^{(1,d)},
\end{align*}
where
\begin{align*}
\left( b_{1,1}^{0,y}\right)^2 = \left( \mathcal{Y}_{1,0}^{(1,d)}\right)^2 + \cdots + \left( \mathcal{Y}_{d,0}^{(1,d)} \right)^2, \:
b_{1,1}^{0,x} = \la x Y_{0,1}^{(1,d)}, Y_{1,1}^{(1,d)} \ra_{\g_{1,d},w} = \frac{\mathcal{X}^{(1,d)}_{1,0}\mathcal{Y}^{(1,d)}_{1,0}}{b_{1,1}^{0,y}},
\end{align*}
and there are (generally nonzero) constants $C^{(1,d)}_{1,2}, \ldots, C^{(1,d)}_{d,2}$ such that
\begin{align*}
Y_{1,2}^{(1,d)} &= \frac{1}{b_{2,1}^{0,x}}\left(xY_{0,1}^{(1,d)} - \mathcal{X}_{0,0}^{(1,d)}Y_{0,1}^{(1,d)} - b_{1,1}^{0,x}Y_{1,1}^{(1,d)}\right) \\
& =C^{(1,d)}_{1,2}P_{1,0}^{(1,d)} + \cdots + C^{(1,d)}_{d,2}P_{d,0}^{(1,d)} =  \mathbf{P}^{(1,d)}\mathbf{C}_{2}^{(1,d)}.
\end{align*}
Hence (cf.~(\ref{eq:C1dinblocks}) and (\ref{eq:c12})), $\CC^{(1,d)}_0 = \bm{e}_0$ and $\CC^{(1,d)}_1 = \left(
\begin{array}{c c}
\mathbf{C}_{1}^{(1,d)} & \mathbf{C}_{2}^{(1,d)}
\end{array}
\right)$, where 
\begin{equation*}
\mathbf{C}_{1}^{(1,d)} = \frac{1}{b_{1,1}^{0,y}}\left( \mathcal{Y}_{1,0}^{(1,d)}\bm{e}_1 + \cdots +  \mathcal{Y}_{d,0}^{(1,d)} \bm{e}_d  \right), \qquad
\mathbf{C}_{2}^{(1,d)} = C^{(1,d)}_{1,2}\bm{e}_1 + \cdots +  C^{(1,d)}_{d,2} \bm{e}_d .
\end{equation*}
This illustrates that the block $\CC^{(1,d)}_1$  in (\ref{eq:C1dinblocks}), obtained by first multiplying by $x$ and orthogonalising, is sparser (has fewer nonzero elements) than the one obtained by first multiplying by $y$. The same conclusion holds true for the case $m = 2$.

For ease of reference, we restate the monomial basis of degree $n$ (given in (\ref{eq:Bgc1}) and (\ref{eq:Bgc3})) on the curves considered in this section:
\begin{align}
\CB_{n}(\g_{m,d}) = \left\{x^{n-k}y^k  \right\}_{k=0}^{n},\quad  0\leq n \leq d-1,  \quad \CB_{n}(\g_{m,d}) = \left\{x^{d-k}y^{n+k-d}  \right\}_{k=1}^{d}, \quad n \geq d. \label{eq:gmdmonob}
\end{align}
To ensure the OPs span the same space as the monomial basis and in light of the example above, to construct $\YY_{n+1}^{(m,d)}$ for $n \leq d-2$, we first multiply by $x$ (and orthogonalise) and then multiply by $y$. That is, $Y_{n+1,k}^{(m,d)}$, $k = 1, \ldots, n+1$  are generated by orthogonalising $xY_{n,k}^{(m,d)}$ and $Y^{(m,d)}_{n+1,n+2}$ is found by orthogonalising $yY^{(m,d)}_{n,n+1}$. We refer to this sequence as the \emph{primary sequence of orthogonalisation}. This implies that for $0 \leq n \leq d-1$, the  expansions of the OPs in the monomial basis (\ref{eq:gmdmonob}) take the form  
\begin{equation}
Y^{(m,d)}_{n,k} =   \a_{k,k}^{(n)}x^{n-k+1}y^{k-1}+ \cdots  + \a_{1,k}^{(n)}x^{n-1}y + \a_{0,k}^{(n)}x^n  + \mathcal{O}(x^{\ell_1}y^{\ell_2}), \qquad 0 \leq \ell_1+\ell_2<n,  \label{eq:monoYmdspan1}
\end{equation}
for $k = 1, \ldots, n+1$, with $\a_{k,k}^{(n)} \neq 0$, 
and where $\mathcal{O}(x^{\ell_1}y^{\ell_2})$ represents the remainder of the monomial  expansion, which consists of terms of degree less than $n$ (hence $\ell_1+\ell_2<n$).
Hence, for the $Y_{n,k}^{(m,d)}$ generated via the primary sequence of orthogonalization, it follows from (\ref{eq:gmdmonob}) that   $\deg_{\g_{m,d}}(Y_{n,k}^{(m,d)}) = n$.

To generate $\YY_d^{(m,d)}$, we note that $Y_{d,1}^{(m,d)}$ cannot be constructed by orthogonalising $xY_{d-1,1}^{(m,d)}$ since $xY_{d-1,1}^{(m,d)}$ has degree less than $d$ on $\g_{m,d}$. This is because its monomial expansion is  $xY_{d-1,1}^{(m,d)} = \a_{0,1}^{(d-1)}x^d +  \mathcal{O}(x^{\ell_1}y^{\ell_2})$ with $\ell_1 + \ell_2 < d$ 
and since $y = \phi$ (if $m = 1$) or $y^2 = \phi$ (if $m = 2$), with $\phi = c_dx^d + \cdots + c_1x + c_0$, $c_d \neq 0$, $\deg_{\g_{m,d}}(x^d) = \deg_{\g_{m,d}}(xY_{d-1,1}^{(m,d)}) < d$. However, $\YY_d^{(m,d)}$ (or, more generally, $\YY_n^{(m,d)}$) can be constructed by orthogonalising  $xY_{n-1,k}^{(m,d)}$, $k = 2, \ldots, d$, and then $yY_{n-1,d}^{(m,d)}$ for $n \geq d$ because these polynomials are linearly independent and it follows from (\ref{eq:gmdmonob}) that $\deg_{\g_{m,d}}(yY_{n-1,d}^{(m,d)}) = \deg_{\g_{m,d}}(xY_{n-1,k}^{(m,d)}) = n$, $k = 2, \ldots, d$. Alternatively, and for the same reasons, we can orthogonalise $yY_{n-1,k}^{(m,d)}$, $k = 1, \ldots, d$ to construct $\YY_n^{(m,d)}$, $n \geq d$. It can be verified (using the commutativity of multiplication by $x$ and $y$ and applying the Gram--Schmidt procedure to the monomial basis) that both these sequences yield the same OP basis.
Hence, we refer to either of these sequences as the \emph{secondary sequence of orthogonalisation}. This  implies that for $n \geq d$, the monomial expansions of the OPs are
\begin{equation}
Y^{(m,d)}_{n,k} =   \a_{k,k}^{(n)} x^{d-k}y^{n-d+k}+ \cdots + \a_{2,k}^{(n)}x^{d-2}y^{n-d+2} + \a_{1,k}^{(n)}x^{d-1}y^{n-d+1} +   \mathcal{O}(x^{\ell_1}y^{\ell_2}),  \label{eq:monoYmdspan2} 
\end{equation}
for $k = 1, \ldots, d$, with $\a_{k,k}^{(n)} \neq 0$, and where $\ell_1 + \ell_2 < n$.


To summarize: the $Y_{n,k}^{(m,d)}$ generated via the primary and secondary sequences of orthogonalization span the same space as the monomial basis (\ref{eq:gmdmonob}) on the curve $\g_{m,d}$ and $\deg_{\g_{m,d}}(Y_{n,k}^{(m,d)}) = n$. Since the $Y_{n,k}^{(m,d)}$ thus generated are orthogonal by construction, they are (orthonormal) OPs on $\g_{m,d}$ with respect to $\la \cdot, \cdot \ra_{\g_{m,d},w}$. The $\YY_n^{(m,d)}$ constructed in this manner are in $\CV_n$ (the space of bivariate OPs of degree $n$ with respect to $\langle \cdot, \cdot \rangle_{\g_{m,d},w}$) but note that  if $\mathcal{C} \in \RR^{d \times d}$ is an orthogonal matrix, then  $\YY_n^{(m,d)} \mathcal{C}$ (for $n \geq d-1$) are also in $\CV_n$ since $\la \YY_n^{(m,d)} \mathcal{C}, \YY_n^{(m,d)} \mathcal{C} \ra_{\g_{m,d},w}  = I \in \RR^{d \times d}$.

We now describe the computation of the connection coefficients and the structure of the connection matrices (examples of which are shown in Figures~\ref{fig:C_m_1} and~\ref{fig:Cexamples_curve}) in more detail. 

\subsection{Recursive formulae for the connection coefficients}  \label{sect:rec4ms}

Let $C^{(m,d)}_{r,c}$, $r, c \geq 0$ denote the entries of the connection matrix. We introduce the map $\ell = \ell(n,k)$, from the indices $(n,k)$ of the OP $Y^{(m,d)}_{n,k}$ for $m = 1, 2$ to the column index of the quasimatrix
\begin{equation*}
\mathbf{Y}^{(m,d)} = \left(
\begin{array}{c  c c  c}
Y^{(m,d)}_{0,1} & Y^{(m,d)}_{1,1} & Y^{(m,d)}_{1,2} & \cdots
\end{array}
\right).
\end{equation*}
The OP $Y^{(m,d)}_{0,1}$ is in column $0$ of $\mathbf{Y}^{(m,d)}$,  $Y^{(m,d)}_{1,1}$ is in column $1$ of $\mathbf{Y}^{(m,d)}$, etc. Hence,
\begin{align}
\ell(n,k) = \sum_{\ell=0}^{n-1}\dim \CV_{\ell} + k-1 = \begin{cases}
\vspace{0.25cm}\displaystyle{\frac{n(n+1)}{2}} + k-1, & 0\leq n \leq d-1,\qquad 1 \leq k \leq n+1 \\
\displaystyle{dn - \frac{d(d - 1)}{2}} + k-1, & n \geq d, \qquad 1 \leq k \leq d 
\end{cases},  \label{eq:ldef} 
\end{align}
where $\dim \CV_{\ell} = \min \{\ell+1, d\}$.  Let $\mathbf{C}^{(m,d)}_{\ell(n,k)}\in \RR^{\infty \times 1}$ denote column number $\ell(n,k) \geq 0$ of the connection matrix $C^{(m,d)}$, then since $\mathbf{Y}^{(m,d)} = \mathbf{P}^{(m,d)}C^{(m,d)}$, we have that
\begin{equation}
Y_{n,k}^{(m,d)} = \mathbf{P}^{(m,d)}\mathbf{C}^{(m,d)}_{\ell(n,k)},  \label{eq:ypc}
\end{equation}
which expresses (\ref{eq:Y1exp})--(\ref{eq:Y2exp}) in quasimatrix notation. For $m = 1$, it follows from (\ref{eq:Y1exp}) and the orthogonality of the $P^{(1,d)}_{k,0}$ polynomials with respect to $\la \cdot, \cdot \ra_{\g_{1,d},w}$ that
\begin{equation}
C^{(1,d)}_{r,\ell(n,k)} = \la Y^{(1,d)}_{n,k}, P^{(1,d)}_{r,0} \ra_{\g_{1,d},w}, \qquad r \geq 0, \label{eq:Cijdef}
\end{equation}
and similarly, for $m = 2$, (\ref{eq:Y2exp}) and the orthogonality of the $P^{(2,d)}_{k,0}$, $k \geq 0$ and $P^{(2,d)}_{k,1}$, $k \geq 1$ polynomials imply that $C^{(2,d)}_{0,0} = \la Y^{(2,d)}_{0,1}, P^{(2,d)}_{0,0} \ra_{\g_{2,d},w}$ and 
\begin{equation}
C^{(2,d)}_{2r-1,\ell(n,k)} = \la Y^{(2,d)}_{n,k}, P^{(2,d)}_{r,0} \ra_{\g_{2,d},w}, \qquad C^{(2,d)}_{2r,\ell(n,k)} = \la Y^{(2,d)}_{n,k}, P^{(2,d)}_{r,1} \ra_{\g_{2,d},w}, \qquad r \geq 1.  \label{eq:C2coeffs}
\end{equation}

Let $a_{i,j}^{n,x}$ for $i, j \geq 1$ denote an entry of $A_{n,x}^{(1,d)}$ or $A_{n,x}^{(2,d)}$, where it should be clear from the context whether $m =1$ or $m = 2$, and similarly, $b_{i,j}^{n,x}$ denotes an entry of $B_{n,x}^{(1,d)}$ or $B_{n,x}^{(2,d)}$. Likewise,   $a_{i,j}^{n,y}$ and  $b_{i,j}^{n,y}$ denote entries of, respectively, $A_{n,y}^{(1,d)}$ or $A_{n,y}^{(2,d)}$ and $B_{n,y}^{(1,d)}$ or $B_{n,y}^{(2,d)}$.

Using the known initial columns of the connection matrices given in Proposition~\ref{prop:Cinit}, we can recursively compute entries in the subsequent columns using the results in Propositions~\ref{prop:C1forms}, \ref{prop:C2forms}, \ref{prop:C1struct} and~\ref{prop:C2struct}


\begin{prop}\label{prop:C1forms}
For $m = 1$, the connection coefficients satisfy the following equations: suppose first that $0 \leq n \leq d-2$, then for $k = 1, \ldots, n+1$
\begin{align}
 b^{n,x}_{k,k}C_{j,\ell(n+1,k)}^{(1,d)} = &\sum_{r=j-1}^{j+1}\mathcal{X}^{(1,d)}_{r,j}C_{r,\ell(n,k)}^{(1,d)} -  \sum_{r=k}^{n}b^{n-1,x}_{k,r}C_{j,\ell(n-1,r)}^{(1,d)} \label{eq:C1}  \\
&  - \sum_{r=1}^{n+1}a^{n,x}_{r,k}C_{j,\ell(n,r)}^{(1,d)} - \sum_{r=1}^{k-1}b^{n,x}_{r,k}C_{j,\ell(n+1,r)}^{(1,d)}, \notag
\end{align}
and we obtain the column $\mathbf{C}^{(1,d)}_{\ell(n+1,n+2)}$ from 
\begin{align}
 \hspace*{-0 cm} b^{n,y}_{n+2,n+1}C_{j,\ell(n+1,n+2)}^{(1,d)} = & \sum_{r=j-d}^{j+d}\mathcal{Y}^{(1,d)}_{r,j}C_{r,\ell(n,n+1)}^{(1,d)} -  b^{n-1,y}_{n+1,n}C_{j,\ell(n-1,n)}^{(1,d)}  \label{eq:C2} \\
& - \sum_{r=1}^{n+1}a^{n,y}_{r,n+1}C_{j,\ell(n,r)}^{(1,d)} - \sum_{r=1}^{n+1}b^{n,y}_{r,n+1}C_{j,\ell(n+1,r)}^{(1,d)};\notag
\end{align}
the entries of $\mathbf{C}^{(1,d)}_{\ell(d,k)}$ for $k = 1, \ldots, d$ follow from
\begin{align}
 b^{d-1,y}_{k,k}C_{j,\ell(d,k)}^{(1,d)} = & \sum_{r=j-d}^{j+d}\mathcal{Y}^{(1,d)}_{r,j}C_{r,\ell(d-1,k)}^{(1,d)} -   \sum_{r=\max\{k-1,1\}}^{d-1}b^{d-2,y}_{k,r}C_{j,\ell(d-2,r)}^{(1,d)} \label{eq:C3} \\
& - \sum_{r=1}^{d}a^{d-1,y}_{r,k}C_{j,\ell(d-1,r)}^{(1,d)} - \sum_{r=1}^{k-1}b^{d-1,y}_{r,k}C_{j,\ell(d,r)}^{(1,d)} \notag
\end{align}
and for  $n \geq d$, we compute $\mathbf{C}^{(1,d)}_{\ell(n+1,k)}$ for $k = 1, \ldots, d$ using
\begin{align}
b^{n,y}_{k,k}C_{j,\ell(n+1,k)}^{(1,d)}  = & \sum_{r=j-d}^{j+d}\mathcal{Y}^{(1,d)}_{r,j}C_{r,\ell(n,k)}^{(1,d)} - \sum_{r=k}^{d}b^{n-1,y}_{k,r}C_{j,\ell(n-1,r)}^{(1,d)} \label{eq:C4} \\
& - \sum_{r=1}^{d}a^{n,y}_{r,k}C_{j,\ell(n,r)}^{(1,d)} - \sum_{r=1}^{k-1}b^{n,y}_{r,k}C_{j,\ell(n+1,r)}^{(1,d)}.  
\end{align}
The coefficients $b^{n,x}_{k,k}$, $b^{n,y}_{n+2,n+1}$, $b^{d-1,y}_{k,k}$ and $b^{n,y}_{k,k}$ on the left-hand sides of, respectively, (\ref{eq:C1})--(\ref{eq:C4}) are determined by the requirement that the OPs have unit norm, e.g., in (\ref{eq:C1}) we require $\left(\mathbf{C}^{(1,d)}_{\ell(n+1,k)} \right)^{\top} \mathbf{C}^{(1,d)}_{\ell(n+1,k)} = 1 = \la Y_{n+1,k}^{(1,d)}, Y_{n+1,k}^{(1,d)} \ra_{\g_{1,d},w}$. 
\end{prop}
\begin{proof}
It follows from the definition (\ref{eq:ABxmd}) that $A_{n,x}^{(1,d)}$ and $A_{n,y}^{(1,d)}$ are dense symmetric matrices, e.g.,
\begin{align}
A_{n,x}^{(1,d)} = 
 \left(
\begin{array}{c c c c c}
a_{1,1}^{n,x} & a_{2,1}^{n,x} & a_{3,1}^{n,x} & \cdots & a_{\dim \CV_n,1}^{n,x} \\
a_{2,1}^{n,x} & a_{2,2}^{n,x} & a_{3,2}^{n,x} & \cdots & a_{\dim \CV_n,2}^{n,x} \\
a_{3,1}^{n,x} & a_{3,2}^{n,x} & a_{3,3}^{n,x}  &\cdots & a_{\dim \CV_n,3}^{n,x} \\
\vdots  & \ddots & \ddots & \ddots& \vdots \\
a_{\dim \CV_n,1}^{n,x} & a_{\dim \CV_n,2}^{n,x} & \cdots & \cdots & a_{\dim \CV_n,\dim \CV_n}^{n,x}
\end{array}  \right) \in \RR^{\dim \CV_n \times \dim \CV_n}, \label{eq:Anx}
\end{align}
where $\dim \CV_n = \dim \CV_n(\g_{1,d},w) = n+1$ for $0 \leq n \leq d-1$ and $\dim \CV_n = d$ for $n \geq d$. The three-term recurrences (\ref{eq:xYmdrec})--(\ref{eq:yYmdrec}) and the monomial expansion (\ref{eq:monoYmdspan1}) imply that for $0 \leq n \leq d-2$:
\begin{equation*}
xY_{n,k}^{(1,d)} \in \mathrm{span}\{\YY_{n-1}^{(1,d)},\YY_n^{(1,d)}, Y_{n+1,1}^{(1,d)}, \ldots, Y_{n+1,k}^{(1,d)}     \}, \qquad 1 \leq k \leq n+1,
\end{equation*}
and
\begin{equation*}
yY_{n,k}^{(1,d)}\in \mathrm{span}\{\YY_{n-1}^{(1,d)},\YY_n^{(1,d)}, Y_{n+1,1}^{(1,d)}, \ldots, Y_{n+1,k+1}^{(1,d)}     \}, \qquad 1 \leq k \leq n+1. 
\end{equation*}
Therefore, the matrices $B_{n,x}^{(1,d)}$ and $B_{n,y}^{(1,d)}$ have the following structures 
for $0 \leq n \leq d-2$:
\begin{align}
B_{n,x}^{(1,d)} = 
\left(
\begin{array}{c c c c}
b_{1,1}^{n,x} & b_{1,2}^{n,x}  & \cdots & b_{1,n+1}^{n,x} \\
 0 & b_{2,2}^{n,x} & \cdots & b_{2,n+1}^{n,x} \\
\vdots  & \ddots & \ddots &  \vdots \\
0 & \cdots & 0 &  b_{n+1,n+1}^{n,x} \\
0 & \cdots & 0  & 0
\end{array}  \right) \in \RR^{n+2 \times n+1},  \label{eq:Bnxc1}
\end{align}
and
\begin{align}
B_{n,y}^{(1,d)} =
\left(
\begin{array}{c c c c}
b_{1,1}^{n,y} & b_{1,2}^{n,y}  & \cdots & b_{1,n+1}^{n,y} \\
b_{2,1}^{n,y}  & b_{2,2}^{n,y} & \cdots & b_{2,n+1}^{n,y} \\
0 & b_{3,2}^{n,y} & \cdots & b_{3,n+1}^{n,y}\\
\vdots  & \ddots & \ddots &  \vdots \\
0 & \cdots & 0 &  b_{n+2,n+1}^{n,y} 
\end{array}  \right) \in \RR^{n+2 \times n+1}.  \label{eq:Bnyc1}
\end{align}

The monomial expansions (\ref{eq:monoYmdspan2}) imply that for $n \geq d-1$, 
\begin{equation*}
xY_{n,k}^{(1,d)} \in \mathrm{span}\{\YY_{n-1}^{(1,d)},\YY_n^{(1,d)}, Y_{n+1,1}^{(1,d)}, \ldots, Y_{n+1,k-1}^{(1,d)}     \}, \qquad 1 \leq k \leq d,
\end{equation*}
and
\begin{equation*}
yY_{n,k}^{(1,d)}\in \mathrm{span}\{\YY_{n-1}^{(1,d)},\YY_n^{(1,d)}, Y_{n+1,1}^{(1,d)}, \ldots, Y_{n+1,k}^{(1,d)}     \}, \qquad 1 \leq k \leq d. 
\end{equation*}
Hence, for $n \geq d-1$, 
\begin{align}
B_{n,x}^{(1,d)} = \left(
\begin{array}{c c c c c}
0 & b_{1,2}^{n,x} & b_{1,3}^{n,x}  & \cdots & b_{1,d}^{n,x} \\
0 & 0 & b_{2,3}^{n,x} & \cdots & b_{2,d}^{n,x} \\
\vdots  & \ddots & \ddots  &\ddots &  \vdots \\
0 & \cdots & 0 &0 &  b_{d-1,d}^{n,x} \\
0 & \cdots & \cdots & 0  & 0
\end{array}  \right) \in \RR^{d \times d}  \label{eq:Bnxc2}
\end{align}
and
\begin{align}
B_{n,y}^{(1,d)} = \left(
\begin{array}{c c c c }
b_{1,1}^{n,y} & b_{1,2}^{n,y}  & \cdots & b_{1,d}^{n,y} \\
0 & b_{2,2}^{n,y} & \cdots & b_{2,d}^{n,y} \\
\vdots  & \ddots & \ddots  &  \vdots \\
0 & \cdots & 0 &  b_{d,d}^{n,y} 
\end{array}  \right) \in \RR^{d \times d}. \label{eq:Bnyc2}
\end{align}

The three-term recurrences (\ref{eq:xYmdrec}) and (\ref{eq:yYmdrec}) satisfied by the OPs and the structures of the matrices given above imply that for $0 \leq n \leq d-2$ and $k = 1, \ldots, n+1$
\begin{align}
xY_{n,k}^{(1,d)} &=   \sum_{r=k}^{n}b^{n-1,x}_{k,r}Y_{n-1,r}^{(1,d)} + \sum_{r=1}^{n+1}a^{n,x}_{r,k}Y_{n,r}^{(1,d)} + \sum_{r=1}^{k}b^{n,x}_{r,k}Y_{n+1,r}^{(1,d)} \label{eq:xY1} \\
yY_{n,k}^{(1,d)} &=   \sum_{r=\max\{k-1,1\}}^{n}b^{n-1,y}_{k,r}Y_{n-1,r}^{(1,d)} + \sum_{r=1}^{n+1}a^{n,y}_{r,k}Y_{n,r}^{(1,d)} + \sum_{r=1}^{k+1}b^{n,y}_{r,k}Y_{n+1,r}^{(1,d)}. \label{eq:yY1}
\end{align}
The recurrences satisfied by the OPs for the cases $n = d-1$, $k = 1, \ldots, d$ and $n \geq d$, $k = 1, \ldots, d$ can be found  similarly using (\ref{eq:Anx})--(\ref{eq:Bnyc2}).

Let $\mathcal{X}^{(1,d)}_{r,c}$ and $\mathcal{Y}^{(1,d)}_{r,c}$, $r, c \geq 0$ denote the entries of $\mathcal{X}^{(1,d)}$ and $\mathcal{Y}^{(1,d)}$, respectively, which are defined in (\ref{eq:xym1ops}). Since $\mathcal{X}^{(1,d)}$ is tridiagonal and $\mathcal{Y}^{(1,d)}$ has bandwidths $(d, d)$, we have for $j \geq 1$ that
\begin{align}
xP_{j,0}^{(1,d)}  = \sum_{r=j-1}^{j+1}\mathcal{X}^{(1,d)}_{r,j}P_{r,0}^{(1,d)}, \qquad yP_{j,0}^{(1,d)}  = \sum_{r=j-d}^{j+d}\mathcal{Y}^{(1,d)}_{r,j}P_{r,0}^{(1,d)}. \label{eq:xyP}
\end{align}
It now follows from (\ref{eq:Cijdef}), (\ref{eq:xY1}) and (\ref{eq:xyP}) that for $0 \leq n \leq d-2$ and $k = 1, \ldots, n+1$
\begin{align*}
& \la xY^{(1,d)}_{n,k}, P^{(1,d)}_{j,0} \ra_{\g_{1,d},w} = \sum_{r=k}^{n}b^{n-1,x}_{k,r}C_{j,\ell(n-1,r)}^{(1,d)} + \sum_{r=1}^{n+1}a^{n,x}_{r,k}C_{j,\ell(n,r)}^{(1,d)} + \sum_{r=1}^{k}b^{n,x}_{r,k}C_{j,\ell(n+1,r)}^{(1,d)}  \\
& = \la Y^{(1,d)}_{n,k},x P^{(1,d)}_{j,0} \ra_{\g_{1,d},w} = \sum_{r=j-1}^{j+1}\mathcal{X}^{(1,d)}_{r,j}C_{r,\ell(n,k)}^{(1,d)},
\end{align*}
from which we obtain (\ref{eq:C1}).
The recurrence (\ref{eq:C2}) can similarly be derived using (\ref{eq:Cijdef}), (\ref{eq:yY1}) and (\ref{eq:xyP}). The remaining equations (\ref{eq:C3})--(\ref{eq:C4}) follow from the recurrences satisfied by the OPs for the cases $n = d-1$, $k = 1, \ldots, d$ and $n \geq d$, $k = 1, \ldots, d$. 
\end{proof}

The formulae from which the connection coefficients can be computed for $m = 2$ are given in Appendix~\ref{sect:ACm2rec}.

\subsection{Structure of the connection matrices}\label{sect:Cstruct}
We now aim to deduce the entries of the connection matrices that are zero, which will specify the ranges over which the row indices should vary in the formulae for the connection coefficients for $m = 1$ (given in (\ref{eq:C1})--(\ref{eq:C4})) and for $m = 2$ (given in Appendix~\ref{sect:ACm2rec}).

\begin{prop}\label{prop:C1struct}
For $m = 1$, if $0 \leq n \leq d-2$, then $
C^{(1,d)}_{j,\ell(n+1,k)} = 0$ for
\begin{align}
 j < n+1 \qquad &\text{and} \qquad 1 \leq  k \leq n+2 \label{eq:m1jb1} \\
  j > nd \qquad &\text{and} \qquad 1 \leq k \leq n \label{eq:m1jb2}  \\
 j > nd + 1 \qquad &\text{and} \qquad k = n+1 \\
 j > (n+1)d \qquad & and \qquad k = n+ 2. \label{eq:m1jb4}
\end{align}
If $n \geq d-1$, then  
$
C^{(1,d)}_{j,\ell(n+1,k)} = 0
$
for
\begin{align}
 j < d-1+d(n+2-d) \qquad & and \qquad 1 \leq  k \leq d   \label{eq:m1jb5}\\
 j > nd \qquad & and \qquad 1 \leq k \leq d-2 \label{eq:m1jb6} \\
  j > nd + 1 \qquad & and \qquad k = d-1 \\
 j > (n+1)d \qquad & and \qquad k = d  \label{eq:m1jb8}
\end{align}
\end{prop}
\begin{proof}
Recall that the (maximal) degree of an OP $Y^{(1,d)}_{n,k}$ on $\g_{1,d}$ is obtained by setting $y = \phi(x)$ and finding the degree of the resulting polynomial, which is also the degree of $Y^{(1,d)}_{n,k}$ in the $P^{(1,d)}_{n,0}$ basis. We denote this maximal degree by $\deg Y^{(1,d)}_{n,k}$, which is in general different from the (minimal) degree of $Y^{(1,d)}_{n,k}$ on $\g_{1,d}$, viz.~$n$, which is denoted by  $\deg_{\g_{1,d}}(Y^{(1,d)}_{n,k})$. 

To determine the degrees of the OPs $Y^{(1,d)}_{n+1,k}$ in the $P^{(1,d)}_{j,0}$ basis, we let $n \mapsto n+1$ in (\ref{eq:monoYmdspan1}) and (\ref{eq:monoYmdspan2}) and note that among the degree $n+1$ monomial terms,  the leading term has the largest maximal degree, which is, for $0 \leq n \leq d-2$,
\begin{equation}
\deg x^{n+2-k}y^{k-1} = n+2-k+ d(k-1), \qquad k = 1, \ldots,  n+2, \label{eq:maxdc1}
\end{equation}
and for $n \geq d - 1$,
\begin{equation}
\deg x^{d-k}y^{n+1+k-d} = d-k + d(n+1+k-d), \qquad k = 1, \ldots, d.  \label{eq:maxdc2}
\end{equation}
All the OPs of degree $n$ ($\YY_n^{(1,d)}$) appear in the recurrences satisfied by the OPs of degree $n+1$ (see (\ref{eq:xYmdrec})--(\ref{eq:yYmdrec})). Among all monomials of degree $n$ on $\g_{1,d}$, $y^n$ has the largest maximal degree, with $\deg y^n = nd$. We conclude from this and (\ref{eq:maxdc1}) that for $0 \leq n \leq d-2$, $\deg Y^{(1,d)}_{n+1,k} = \max\{nd, n+2-k+ d(k-1)  \}$ and from (\ref{eq:maxdc2}) it follows that for $ n \geq d-1$, $\deg Y^{(1,d)}_{n+1,k} = \max\{nd, d-k + d(n+1+k-d)  \}$. Hence,
\begin{align*}
\deg Y^{(1,d)}_{n+1,k} = \begin{cases}
nd & \text{for } 0 \leq n \leq d-2, 1\leq  k \leq n \text{ and } n \geq d-1,1 \leq k \leq d-2 \\
nd + 1 & \text{for } 0 \leq n \leq d-2, k = n+1 \text{ and } n \geq d-1, k =d-1  \\
(n+1)d & \text{for } 0 \leq n \leq d-2, k = n+2 \text{ and } n \geq d-1, k =d
\end{cases}.
\end{align*}
Since these maximal degrees are also the degrees of the $Y^{(1,d)}_{n+1,k}$ in the $P^{(1,d)}_{j,0}$ basis, they are the upper bounds on $j$ in the expansions $Y^{(1,d)}_{n+1,k} = \sum_{j \geq 0} C^{(1,d)}_{j,\ell(n+1,k)} P^{(1,d)}_{j,0}$. Thus, we have shown that $C^{(1,d)}_{j,\ell(n+1,k)}= 0$ for the cases (\ref{eq:m1jb2})--(\ref{eq:m1jb4}) and (\ref{eq:m1jb6})--(\ref{eq:m1jb8}).

Note from (\ref{eq:maxdc1})--(\ref{eq:maxdc2}) that the smallest maximal degree of a  monomial of degree $n+1$ on $\g_{1,d}$ is attained for $k = 1$, which is $n+1$ for $0 \leq n \leq d-2$ and $d-1 + d(n+2-d)$ for $n \geq d-1$. This implies that
\begin{equation}
\deg_{\g_{1,d}}(P^{(1,d)}_{j,0}) < n+ 1 \text{ for } \begin{cases}
j < n+1 & \text{if } 0 \leq n \leq d-2\\
 j < d-1 + d(n+2-d) & \text{if }  n \geq d-1
\end{cases}.  \label{eq:degp1}
\end{equation} 
Since the  $Y^{(1,d)}_{n+1,k}$ are OPs of degree $n+1$ on $\g_{1,d}$, it follows that  
\begin{equation}
C^{(1,d)}_{j,\ell(n+1,k)} = \la  Y^{(1,d)}_{n+1,k},  P^{(1,d)}_{j,0} \ra_{\g_{1,d},w} = 0 \quad \text{if} \quad \deg_{\g_{1,d}}(P^{(1,d)}_{j,0}) < n+ 1.  \label{eq:c1j0}
\end{equation}
Hence, combining (\ref{eq:degp1}) and (\ref{eq:c1j0}), we have shown that $C^{(1,d)}_{j,\ell(n+1,k)} = 0$ for the cases (\ref{eq:m1jb1}) and (\ref{eq:m1jb5}).
\end{proof}

For $m = 2$, the indices for which $C^{(2,d)}_{j,\ell(n+1,k)} = 0$ are given in Appendix~\ref{Asect:C20}.

\subsection{Numerical stability and computational complexity} The recursive formulae for the connection coefficients given in section~\ref{sect:rec4ms} are numerically unstable. This is to be expected since the formulae arise from the Gram--Schmidt procedure, which is well-known to be unstable~\cite{bjorck} and the instability of the Lanczos algorithm is also well-established (see \cite[section 2.2.3]{Gautschi} and \cite[Supplement D]{chen}). We shall consider the computational cost of our Lanczos (or, equivalently, Gram--Schmidt) algorithm when it is stabilized via  reorthogonalization~\cite{bjorck,chen} (the Lanczos algorithm can also be stabilized if it is performed via Householder transformations~\cite{walker} or Givens rotations~\cite{gragg}).

Due to the instability of the Lanczos procedure, the OPs are not exactly orthogonal in floating point arithmetic. After a new OP has been generated via a new column of the connection matrix (using the formulae in section~\ref{sect:rec4ms} and Appendix~\ref{sect:ACm2rec}), we orthogonalize this OP against lower degree OPs to ensure orthogonality up to the level of roundoff errors. That is, we replace the OP $Y_{n+1,k}^{(m,d)}$ with its re-orthogonalized version:
\begin{equation}
Y_{n+1,k}^{(m,d)} \gets  Y_{n+1,k}^{(m,d)} - \sum_{\substack{0 \leq s \leq n+1\\
1 \leq i \leq \dim\CV_{s}}}\langle Y_{n+1,k}^{(m,d)},   Y^{(m,d)}_{s,i} \rangle_{\g_{m,d},w}\, Y^{(m,d)}_{s,i},  \label{eq:Yreorth}
\end{equation}
which we express in terms of the columns of the connection matrix by setting $Y_{n,k}^{(m,d)} =\mathbf{P}^{(m,d)} \mathbf{C}^{(m,d)}_{\ell(n,k)}$:
\begin{equation*}
\mathbf{C}^{(m,d)}_{\ell(n+1,k)}  \gets \mathbf{C}^{(m,d)}_{\ell(n+1,k)}  - \sum_{\substack{0 \leq s \leq n+1\\
1 \leq i \leq \dim\CV_{s}}} \left[\left( \mathbf{C}^{(m,d)}_{\ell(n+1,k)}  \right)^{\top}\mathbf{C}^{(m,d)}_{\ell(s,i)} \right]\mathbf{C}^{(m,d)}_{\ell(s,i)}.
\end{equation*}
We can use the results in section~\ref{sect:Cstruct} and Appendix~\ref{Asect:C20} (i.e., the fact that the connection matrices are banded) to show that we only need to re-orthogonalize against a constant number of lower degree OPs (as opposed to \emph{all} lower degree OPs). For example, for $m=1$, we have
\begin{equation*}
Y^{(1,d)}_{n+1,k} = \sum_{j = j_{\min}}^{j_{\max}} C^{(1,d)}_{j,\ell(n+1,k)}P^{(1,d)}_{j,0},
\end{equation*}
where $j_{\min} = j_{\min}(n+1,k)$ and $j_{\max} = j_{\max}(n+1,k)$  are given in Proposition~\ref{prop:C1struct}. For instance, for $n \geq d-1$ and $1 \leq k \leq d$, $j_{\min}(n+1,k) = d-1+d(n+2-d)$ and for $k = d$, $j_{\max}(n+1,d) = (n+1)d$.  Likewise, for $m = 2$, in (\ref{eq:y2epxs}), $j$ ranges from $j_{\min}$ to $j_{\max}$, which are given in Appendix~\ref{Asect:C20}. Hence, if in (\ref{eq:Yreorth}) we have $j_{\min}(n+1,k) > j_{\max}(s,i)$, then $\langle Y_{n+1,k}^{(m,d)},   Y^{(m,d)}_{s,i} \rangle_{\g_{m,d},w}$ is identically zero, also in floating point arithmetic, and therefore we only need to re-orthogonalize against OPs $Y^{(m,d)}_{s,i}$ with $j_{\min}(n+1,k) \leq j_{\max}(s,i)$. For $m = 1$, we have to re-orthogonalise $Y^{(m,d)}_{n+1,k}$ against $Y^{(m,d)}_{s,i}$ with $n-d+3 \leq s \leq n+1$ because
\begin{align}
\begin{split}
& \min\left\{ s \in \NN_0 \: :  j_{\min}(n+1,k) \leq j_{\max}(s,i) \right\} \\
&  =\min\left\{ s \in \NN_0 \: : \: d-1+d(n+2-d) \leq s d  \right\} = n - d + 3
\end{split}\label{eq:reorthm1}
\end{align}
and for $m = 2$,  $n-d+4 \leq s \leq n+1$ since
\begin{align}
\begin{split}
& \min\left\{ s \in \NN_0 \: :  j_{\min}(n+1,k) \leq j_{\max}(s,i) \right\} \\
&  =\min\left\{ s \in \NN_0 \: : \: 2d + d(n+1 - d) \leq s d - 1  \right\} = n - d + 4.
\end{split}  \label{eq:reorthm2}
\end{align}

\begin{prop}
For $m = 1, 2$, the connection coefficients and the entries of the Jacobi matrices for OPs of degrees $0, \ldots, N$ (i.e., $\CC^{(m,d)}_n$, $A_{n,x}^{(m,d)}$, $A_{n,y}^{(m,d)}$ for $0 \leq n \leq N$ and $B_{n,x}^{(m,d)}$, $B_{n,y}^{(m,d)}$ for  $0 \leq n \leq N-1$) can be computed in $\mathcal{O}(Nd^4)$ operations without re-orthogonalisation and in $\mathcal{O}(Nd^5)$ operations with re-orthogonalisation.
\end{prop}
\begin{proof}
Starting with the first few columns of $C^{(m,d)}$ given in Proposition~\ref{prop:Cinit}, the subsequent columns of the connection matrix are computed using (for $m= 1$) (\ref{eq:C1})--(\ref{eq:C4}), with $j$ ranging between the bounds given in (\ref{eq:m1jb1})--(\ref{eq:m1jb8}) and (for $m = 2$) using the formulae in  Appendix~\ref{sect:ACm2rec}, with the row indices ranging between the bounds given in Appendix~\ref{Asect:C20}.

To compute the entries of $A_{n,x}^{(m,d)}$,  $B_{n,x}^{(m,d)}$, $A_{n,y}^{(m,d)}$ and $B_{n,y}^{(m,d)}$, we use the fact that
\begin{align}
&\la xY_{n_1,i_1}^{(m,d)}, Y_{n_2,i_2}^{(m,d)} \ra_{\g_{m,d},w} =
\left(\mathbf{C}^{(m,d)}_{\ell(n_1,i_1)}\right)^{\top}\mathcal{X}^{(m,d)}\mathbf{C}^{(m,d)}_{\ell(n_2,i_2)}  
= \left(\mathbf{C}^{(m,d)}_{\ell(n_2,i_2)}\right)^{\top}\mathcal{X}^{(m,d)}\mathbf{C}^{(m,d)}_{\ell(n_1,i_1)}, \label{eq:C1ips}
\end{align}
which follows from (\ref{eq:ypc}), the orthonormality of the $P$-polynomials and the symmetry of $\mathcal{X}^{(m,d)}$. A similar identity holds for inner products of the form $\la yY_{n_1,i_1}^{(m,d)}, Y_{n_2,i_2}^{(m,d)} \ra_{\g_{m,d},w}$.

The most computationally expensive part of the procedure is the computation of $\mathcal{Y}^{(m,d)}\mathbf{C}^{(m,d)}_{\ell(n,k)}$.  For $m = 1$, since the indices of the nonzero rows of $\mathbf{C}^{(1,d)}_{\ell(n,k)}$ range from $d-1+d(n-d+1)$ to $nd$, see (\ref{eq:m1jb5}) and (\ref{eq:m1jb8}) with $n+1\mapsto n$, the column $\mathbf{C}^{(1,d)}_{\ell(n,k)}$ has $\mathcal{O}(d^2)$ nonzero entries. Since $\mathcal{Y}^{(1,d)})$ has bandwidths $(d,d)$, the computation of $\mathcal{Y}^{(1,d)}\mathbf{C}^{(1,d)}_{\ell(n,k)}$ requires $\mathcal{O}(d^3)$ operations, which has to be computed $\mathcal{O}(Nd)$ times, hence the computational complexity is $\mathcal{O}(Nd^4)$ without re-orthogonalisation. By the same reasoning, for $m = 2$ the computational complexity is also $\mathcal{O}(Nd^4)$, however the computational cost is reduced by approximately a factor of $4$ since $\mathcal{Y}^{(2,d)}$ and  $\mathbf{C}^{(2,d)}_{\ell(n,k)}$ have roughly half the number of nonzero entries compared to $\mathcal{Y}^{(1,d)}$ and  $\mathbf{C}^{(1,d)}_{\ell(n,k)}$.

The re-orthogonalisation step (see (\ref{eq:reorthm1})--(\ref{eq:reorthm2})) costs $\mathcal{O}(d^4)$ operations per OP for $m = 1, 2$ and hence the overall computational cost of our Lanczos procedure with re-orthogonalisation is  $\mathcal{O}(Nd^5)$.
\end{proof}

The linear computational complexity (in $N$) of the Lanczos method applied to an orthogonal basis compares favourably to the cubic $\mathcal{O}(N^3d^3)$ complexity (without re-orthogonalisation) of orthogonalising the bivariate monomial bases\footnote{\label{foot:chol}One orthogonalisation method is to compute the Cholesky factor $L$ of the Gram matrix $\mathcal{M}$,  where $\mathcal{M} = \langle \mathbf{M}, \mathbf{M} \rangle_{\g_{m,d},w} = LL^{\top}$, and where $\mathbf{M}$ is a quasimatrix whose columns are the monomial basis functions in (\ref{eq:gmdmonob}). Then  the OPs are given by $\mathbf{Y}^{(m,d)} = \mathbf{M} L^{-\top}$. For OPs of degrees $0, \ldots, N$ this has $\mathcal{O}(N^3d^3)$ complexity since it requires the Cholesky decomposition of a $\nu \times \nu$ finite section of the dense matrix $\mathcal{M}$ with $\nu = \mathcal{O}(Nd)$. } in (\ref{eq:gmdmonob}).

We remark that if the results of Propositions~\ref{prop:C1forms}, \ref{prop:C2forms}, \ref{prop:C1struct} and~\ref{prop:C2struct} are used to compute the connection coefficients for the cases in Proposition~\ref{prop:exOPs}, then the same OP bases as in (\ref{eq:m1d1Ops})--(\ref{eq:evenquarOPs}) are obtained.

\subsection{Numerical experiments}

For $m = 1$, if we compare the upper and lower bounds on $j$ in Proposition~\ref{prop:C1struct} to  the connection matrices shown in Figure~\ref{fig:C_m_1}, we find that the bounds for the cases (\ref{eq:m1jb2})--(\ref{eq:m1jb4}) and (\ref{eq:m1jb6})--(\ref{eq:m1jb8}) are sharp, while the bounds for the cases (\ref{eq:m1jb1}) and (\ref{eq:m1jb5}) are sharp only for $k = 1$. That is, the entries of the first three connection matrices in Figure~\ref{fig:C_m_1} are nonzero for the following rows and columns: for $0 \leq n \leq d-2$, $C^{(1,d)}_{j,\ell(n+1,k)} \neq 0$ for
\begin{align*}
n +1 \leq j \leq nd \qquad & \text{and} \qquad k = 1\\
n +2 \leq  j \leq nd \qquad & \text{and} \qquad 2 \leq k \leq n\\
n +2 \leq  j \leq nd+1 \qquad & \text{and} \qquad k = n+1\\
n +2 \leq  j \leq (n+1)d \qquad & \text{and} \qquad k = n+2
\end{align*}
and for $n \geq d-1$, $C^{(1,d)}_{j,\ell(n+1,k)} \neq 0$ for
\begin{align}
 d-1+d(n+2-d) \leq j \leq nd\qquad & \text{and} \qquad  k = 1 \label{eq:m1ubw} \\
 2d-2+d(n+2-d) \leq  j \leq nd\qquad & \text{and} \qquad  2 \leq k \leq d-2 \notag \\
 2d-1+d(n+2-d) \leq  j \leq nd+1 \qquad & \text{and} \qquad  k = d-1 \notag \\
  2d-1+d(n+2-d) \leq  j \leq (n+1)d \qquad & \text{and} \qquad  k = d \label{eq:m1lbw}
\end{align}

These bounds on the indices of the nonzero entries of $C^{(1,d)}$ imply that the maximum bandwidths of $C^{(1,d)}$ are attained at the entries $C^{(1,d)}_{d-1+d(n+2-d),\ell(n+1,1)}$ (see (\ref{eq:m1ubw})) and $C^{(1,d)}_{(n+1)d,\ell(n+1,d)}$ (see (\ref{eq:m1lbw})). Therefore the bandwidths of $C^{(1,d)}$ are $(\lambda, \mu)$ with $\lambda = \mu = \lambda_1$, where 
\begin{equation*}
\lambda_1 = \frac{(d-2)(d-1)}{2}
\end{equation*}
because 
\begin{equation*}
\lambda = \ell(n+1,1) - \left[d-1+d(n+2-d)\right] = \frac{(d-2)(d-1)}{2},
\end{equation*}
where $\ell$ is defined in (\ref{eq:ldef}), and
\begin{equation*}
\mu = (n+1)d -  \ell(n+1,d)  = \frac{(d-2)(d-1)}{2}.
\end{equation*}

Similar to the case $m = 1$, for $m = 2$ we find that in the connection matrices in Figure~\ref{fig:Cexamples_curve}, the lower bounds on the row indices $j$ with $C^{(2,d)}_{j,\ell(n+1,k)} = 0$ given in Appendix~\ref{Asect:C20} are sharp, while the upper bounds (\ref{eq:m2jub1}), (\ref{eq:m2jub2}) and (\ref{eq:m2jub3}) are sharp only in the case $k = 1$. The maximum bandwidths are therefore attained at the connection coefficients corresponding to $P^{(2,d)}_{j,0}$ with $j = (n+1)d/2$ (which is the largest among the maximal degrees in (\ref{eq:m2monomaxdc2})--(\ref{eq:m2monomaxdc3}), which is attained for $k = d$ and $n$ odd) and  $P^{(2,d)}_{j,1}$ with $j = d+ d(n+1-d)/2$   (which is the smallest among the maximal degrees in (\ref{eq:m2monomaxdc2})--(\ref{eq:m2monomaxdc3}), which is attained for $k = 1$, with $\mathrm{mod}(n+1-d,2)=0$). That is, the maximum bandwidths are attained at  $C_{(n+1)d-1,\ell(n+1,d)}^{(2,d)}$  and $C_{2d + d(n+1 - d),\ell(n+1,1)}^{(2,d)}$.
Hence the bandwidths of the  connection matrices in Figure~\ref{fig:Cexamples_curve}  are $(\lambda, \mu) = (\lambda_2, \lambda_2)$, where
\begin{equation*}
\lambda_2 =  \frac{d(d-3)}{2}
\end{equation*}
This is because the number of rows from the main diagonal to the former entry is
\begin{equation*}
(n+1)d - 1 - \ell(n+1,d) = \frac{d(d-3)}{2},
\end{equation*}
and for the latter entry, the number is
\begin{equation*}
\ell(n+1,1) - \left[2d + d(n+1 - d)   \right] = \frac{d(d-3)}{2}.
\end{equation*}



The bottom-right frame of Figures~\ref{fig:C_m_1} and~\ref{fig:Cexamples_curve} show that if $w$ and $\phi$ are even functions, then the connection matrix has roughly half the number of nonzero entries compared to the case where either $w$ or $\phi$ is not even. This is because the multiplication matrices $\mathcal{X}^{(m,d)}$ and $\mathcal{Y}^{(m,d)}$ have roughly half the number of nonzero entries for even $w$ and $\phi$.

\section{Conclusion}\label{sect:conc}
\setcounter{equation}{0}

For the cases $m = 1, 2$, we have constructed bivariate OPs on the algebraic curves $y^m = \phi$. For $m > 2$, if we generalise the $P$-polynomials 
\begin{equation*}
P^{(m,d)}_{n,k} = y^k p_{n-k}(\phi^{2k/m} w;x), \qquad k = 0, \ldots, m-1,
\end{equation*}
and the inner product 
\begin{align*}
  \la f,g \ra_{\g_{m,d},w} = \begin{cases}  
    \displaystyle{\int}  f\left(x, \sqrt[m]{\phi(x)} \right) g\left(x, \sqrt[m]{\phi(x)} \right)  w(x) \d x,& m \text{ odd}, \\ 
   \displaystyle{\int} \left[  f\left(x, \sqrt[m]{\phi(x)} \right) g\left(x, \sqrt[m]{\phi(x)} \right) \right. &  \\
       \left.\:\:\: + f\left(x, - \sqrt[m]{\phi(x)} \right) g\left(x, -\sqrt[m]{\phi(x)} \right)\right] w(x) \d x,       
       & m \text{ even}, 
      \end{cases}
\end{align*} 
then the $P$-polynomials are not orthogonal with respect to $\la \cdot ,\cdot \ra_{\g_{m,d},w}$ for $m > 2$. We can nevertheless still construct the OPs on $y^m = \phi$ as linear combinations of the $P$-polynomials. However, the computation of the connection matrix would require the Cholesky factorisation of a dense Gram matrix, which has cubic complexity, see footnote~\ref{foot:chol}. For $m = 1, 2$, we showed  that the connection matrix is banded but for $m > 2$, the connection matrix will in general be upper triangular.

Topics for future research include the development of the theory of interpolation, quadrature and orthogonal series on the curves $y^m = \phi$, as was done for quadratic and cubic curves in~\cite{OX2,FOX}. As an application, OPs on $y^m = \phi(x)$ could be used to approximate functions with algebraic singularities of the form $\sqrt[ m]{\phi(x)}$ at the zeros of $\phi$, again as was done in~\cite{OX2,FOX} for square-root and cubic singularities. Another application of the OPs  on $y^m = \phi$ is the construction of OPs inside $y^m = \phi$ and their associated surfaces of revolution for the purpose of spectral methods for partial differential equations. 

%
%
%
%
%
%
%
%
%


\begin{appendix}

\section{Recurrence relations for the connection coefficients for $m = 2$}\label{sect:ACm2rec}
\setcounter{equation}{0}


\begin{prop} \label{prop:C2forms}
For $m = 2$, the connection coefficients satisfy the following equations: suppose that $0 \leq n \leq d-2$, then for $k = 1, 3, \ldots, n+1$
\begin{align*}
 & \hspace{-0 cm} b^{n,x}_{k,k}C^{(2,d)}_{2j-1,\ell(n+1,k)} = \b_{j-1}(w)C^{(2,d)}_{2j-3,\ell(n,k)}  + \a_{j}(w) C^{(2,d)}_{2j-1,\ell(n,k)} + \b_{j}(w)C^{(2,d)}_{2j+1,\ell(n,k)}  \\
 & - \sum_{\substack{r=k \\ \mathrm{mod}(r,2) = \mathrm{mod}(k,2)}}^{n}b^{n-1,x}_{k,r}C^{(2,d)}_{2j-1,\ell(n-1,r)} - \sum_{\substack{r > 0 \\ \mathrm{mod}(r,2)=\mathrm{mod}(k,2)}}^{n+1}a^{n,x}_{r,k}C^{(2,d)}_{2j-1,\ell(n,r)} \notag \\
 & - \sum_{\substack{r > 0 \\ \mathrm{mod}(r,2)=\mathrm{mod}(k,2)}}^{k-2}b^{n,x}_{r,k}C^{(2,d)}_{2j-1,\ell(n+1,r)};  \notag
\end{align*}
and for $k = 2, 4, \ldots, n+1$,
\begin{align*}
 & \hspace{-0 cm} b^{n,x}_{k,k}C^{(2,d)}_{2j,\ell(n+1,k)} = \b_{j-2}(\phi w)C^{(2,d)}_{2j-2,\ell(n,k)}  + \a_{j-1}(\phi w) C^{(2,d)}_{2j,\ell(n,k)} + \b_{j-1}(\phi w)C^{(2,d)}_{2j+2,\ell(n,k)}  \notag \\
 & - \sum_{\substack{r=k \\ \mathrm{mod}(r,2) = \mathrm{mod}(k,2)}}^{n}b^{n-1,x}_{k,r}C^{(2,d)}_{2j,\ell(n-1,r)} 
  - \sum_{\substack{r > 0 \\ \mathrm{mod}(r,2)=\mathrm{mod}(k,2)}}^{n+1}a^{n,x}_{r,k}C^{(2,d)}_{2j,\ell(n,r)} \notag \\
 & - \sum_{\substack{r > 0 \\ \mathrm{mod}(r,2)=\mathrm{mod}(k,2)}}^{k-2}b^{n,x}_{r,k}C^{(2,d)}_{2j,\ell(n+1,r)}. \notag
\end{align*}
Suppose that $0 \leq n \leq d-2$ and $k = n + 1$, if $n$ is odd
\begin{align*}
& \hspace{-0 cm} b^{n,y}_{k+1,k}C_{2j-1,\ell(n+1,k+1)}^{(2,d)}  = \sum_{r = j-d+1}^{j+1} r_{r-1,j} C_{2r,\ell(n,k)}^{(2,d)}   -\sum_{\substack{r=\max\{k-1,1\}\\ \mathrm{mod}(r,2) = \mathrm{mod}(k-1,2)}}^{n}b^{n-1,y}_{k,r}C_{2j-1,\ell(n-1,r)}^{(2,d)}  \\
&  - \sum_{\substack{r>0 \\ \mathrm{mod}(r,2)=\mathrm{mod}(k-1,2) }}^{n+1}a^{n,y}_{r,k}C_{2j-1,\ell(n,r)}^{(2,d)} - \sum_{\substack{r>0 \\ \mathrm{mod}(r,2)=\mathrm{mod}(k-1,2) }}^{k-1}b^{n,y}_{r,k}C_{2j-1,\ell(n+1,r)}^{(2,d)} \notag
\end{align*}
and if $n$ is even
\begin{align*}
& \hspace{-0 cm} b^{n,y}_{k+1,k}C_{2j,\ell(n+1,k+1)}^{(2,d)}  = \sum_{r = j-1}^{j-1+d} r_{j-1,r} C_{2r-1,\ell(n,k)}^{(2,d)}   -\sum_{\substack{r=\max\{k-1,1\}\\ \mathrm{mod}(r,2) = \mathrm{mod}(k-1,2)}}^{n}b^{n-1,y}_{k,r}C_{2j,\ell(n-1,r)}^{(2,d)} \\ 
&  - \sum_{\substack{r>0 \\ \mathrm{mod}(r,2)=\mathrm{mod}(k-1,2) }}^{n+1}a^{n,y}_{r,k}C_{2j,\ell(n,r)}^{(2,d)} - \sum_{\substack{r>0 \\ \mathrm{mod}(r,2)=\mathrm{mod}(k-1,2) }}^{k-1}b^{n,y}_{r,k}C_{2j,\ell(n+1,r)}^{(2,d)}.  \notag
\end{align*}
Suppose that $n = d-1$, then for $k = 2, 4, \ldots, d$
\begin{align*}
& \hspace{-0 cm} b^{n,y}_{k,k}C_{2j-1,\ell(n+1,k)}^{(2,d)}  = \sum_{r = j-d+1}^{j+1} r_{r-1,j} C_{2r,\ell(n,k)}^{(2,d)}   -\sum_{\substack{r=\max\{k-1,1\}\\ \mathrm{mod}(r,2) = \mathrm{mod}(k-1,2)}}^{d-1}b^{n-1,y}_{k,r}C_{2j-1,\ell(n-1,r)}^{(2,d)} \\
&  - \sum_{\substack{r>0 \\ \mathrm{mod}(r,2)=\mathrm{mod}(k-1,2) }}^{d}a^{n,y}_{r,k}C_{2j-1,\ell(n,r)}^{(2,d)} - \sum_{\substack{r>0 \\ \mathrm{mod}(r,2)=\mathrm{mod}(k,2) }}^{k-2}b^{n,y}_{r,k}C_{2j-1,\ell(n+1,r)}^{(2,d)}  \notag
\end{align*}
and for $k = 1, 3, \ldots, d$
\begin{align*}
& \hspace{-0 cm} b^{n,y}_{k,k}C_{2j,\ell(n+1,k)}^{(2,d)}  = \sum_{r = j-1}^{j-1+d} r_{j-1,r} C_{2r-1,\ell(n,k)}^{(2,d)}   -\sum_{\substack{r=\max\{k-1,1\}\\ \mathrm{mod}(r,2) = \mathrm{mod}(k-1,2)}}^{d-1}b^{n-1,y}_{k,r}C_{2j,\ell(n-1,r)}^{(2,d)}\\
&  - \sum_{\substack{r>0 \\ \mathrm{mod}(r,2)=\mathrm{mod}(k-1,2) }}^{d}a^{n,y}_{r,k}C_{2j,\ell(n,r)}^{(2,d)} - \sum_{\substack{r>0 \\ \mathrm{mod}(r,2)=\mathrm{mod}(k,2) }}^{k-2}b^{n,y}_{r,k}C_{2j,\ell(n+1,r)}^{(2,d)}.\notag
\end{align*}
Suppose that $n \geq d$. If $\mathrm{mod}(n+1-d,2)=1$, then for $k = 1, 3, \ldots, d$
\begin{align*}
& \hspace{-0 cm} b^{n,y}_{k,k}C_{2j-1,\ell(n+1,k)}^{(2,d)}  = \sum_{r = j-d+1}^{j+1} r_{r-1,j} C_{2r,\ell(n,k)}^{(2,d)}   -\sum_{\substack{r= k\\ \mathrm{mod}(r,2) = \mathrm{mod}(k,2)}}^{d}b^{n-1,y}_{k,r}C_{2j-1,\ell(n-1,r)}^{(2,d)} \\
&  - \sum_{\substack{r>0 \\ \mathrm{mod}(r,2)=\mathrm{mod}(k-1,2) }}^{d}a^{n,y}_{r,k}C_{2j-1,\ell(n,r)}^{(2,d)} - \sum_{\substack{r>0 \\ \mathrm{mod}(r,2)=\mathrm{mod}(k,2) }}^{k-2}b^{n,y}_{r,k}C_{2j-1,\ell(n+1,r)}^{(2,d)}   \notag
\end{align*}
and this equation also holds if $\mathrm{mod}(n+1-d,2)=0$ and for $k = 2, 4, \ldots, d$. If $\mathrm{mod}(n+1-d,2)=0$, then for $k = 1, 3, \ldots, d$
\begin{align*}
& \hspace{-0 cm} b^{n,y}_{k,k}C_{2j,\ell(n+1,k)}^{(2,d)}  = \sum_{r = j-1}^{j-1+d} r_{j-1,r} C_{2r-1,\ell(n,k)}^{(2,d)}   -\sum_{\substack{r=k \\ \mathrm{mod}(r,2) = \mathrm{mod}(k,2)}}^{d}b^{n-1,y}_{k,r}C_{2j,\ell(n-1,r)}^{(2,d)} \\ 
&  - \sum_{\substack{r>0 \\ \mathrm{mod}(r,2)=\mathrm{mod}(k-1,2) }}^{d}a^{n,y}_{r,k}C_{2j,\ell(n,r)}^{(2,d)} - \sum_{\substack{r>0 \\ \mathrm{mod}(r,2)=\mathrm{mod}(k,2) }}^{k-2}b^{n,y}_{r,k}C_{2j,\ell(n+1,r)}^{(2,d)} \notag
\end{align*}
and this equation also holds if $\mathrm{mod}(n+1-d,2)=1$ and for $k = 2, 4, \ldots, d$.
\end{prop}
\begin{proof}
Note from (\ref{eq:YPinm2}), (\ref{eq:C2dindblocks}) and (\ref{eq:c25}) that for $m = 2$, each of the OPs of degrees $0$, $1$ and $2$ can be expressed  as a linear combination of either (i) the $P^{(2,d)}_{n,0}$ polynomials or (ii) the $P^{(2,d)}_{n,1}$ polynomials. In the former case, it is expressed as a polynomial in $x$ and in the latter, as a $y$ times a polynomial in $x$. Since OPs of degree $n+1$ are generated by orthogonalizing either $xY_{n,k}^{(2,d)}$ or  $yY_{n,k}^{(2,d)}$ for $n \geq 2$ and $y^2 = \phi(x)$, it follows that every $Y_{n,k}^{(2,d)}$ can be expanded in terms of either the $P^{(2,d)}_{n,0}$ or the $P^{(2,d)}_{n,1}$ polynomials. In the former case, $C^{(2,d)}_{2j,\ell(n,k)} = 0$, $j \geq 1$  and in the latter, $C^{(2,d)}_{2j-1,\ell(n,k)} = 0$, $j \geq 1$ (see (\ref{eq:Y2exp}) and (\ref{eq:C2coeffs})). We also note that in the former case, the monomial expansion of $Y_{n,k}^{(2,d)}$ contains only even powers of $y$ and in the latter case, only odd powers of $y$. 

From the  monomial expansions (\ref{eq:monoYmdspan1}), for $0 \leq n \leq d-1$, we conclude (since $\a_{k,k}^{(n)}\neq 0$) that if $k$ is odd, then $Y_{n,k}^{(2,d)}$ has an expansion in the $P^{(2,d)}_{n,0}$ polynomials, and hence $C^{(2,d)}_{2j,\ell(n,k)} = 0$, $j \geq 1$. Using similar reasoning and also (\ref{eq:monoYmdspan2}), we list all the cases:
\begin{equation}  \label{eq:y2epxs}
\begin{split}
& 1 \leq n \leq d-1, \qquad \mathrm{mod}(k,2)=1, \qquad Y_{n,k}^{(2,d)}(x,y)  = \sum C^{(2,d)}_{2j-1,\ell} P_{j,0}^{(2,d)}  \\
& 1 \leq n \leq d-1, \qquad \mathrm{mod}(k,2)=0, \qquad  Y_{n,k}^{(2,d)}(x,y)  = \sum C^{(2,d)}_{2j,\ell} P_{j,1}^{(2,d)}  \\
& n \geq d, \qquad \mathrm{mod}(n-d,2)=0, \mathrm{mod}(k,2)=1, \qquad  Y_{n,k}^{(2,d)}(x,y)  = \sum C^{(2,d)}_{2j,\ell} P_{j,1}^{(2,d)} \\
& n \geq d, \qquad \mathrm{mod}(n-d,2)=0, \mathrm{mod}(k,2)=0, \qquad  Y_{n,k}^{(2,d)}(x,y)  = \sum C^{(2,d)}_{2j-1,\ell} P_{j,0}^{(2,d)} \\
& n \geq d, \qquad \mathrm{mod}(n-d,2)=1, \mathrm{mod}(k,2)=1, \qquad  Y_{n,k}^{(2,d)}(x,y)  = \sum C^{(2,d)}_{2j-1,\ell} P_{j,0}^{(2,d)} \\
& n \geq d, \qquad \mathrm{mod}(n-d,2)=1, \mathrm{mod}(k,2)=0, \qquad  Y_{n,k}^{(2,d)}(x,y)  = \sum C^{(2,d)}_{2j,\ell} P_{j,1}^{(2,d)}
\end{split}
\end{equation}
where $j \geq 1$, $\ell = \ell(n,k)$; in the first two cases, $1 \leq k \leq n+1$ and in the final four cases, $1 \leq k \leq d$.

Observe that if $Y_{n,k}^{(2,d)}$ has an expansion in the $P_{j,0}^{(2,d)}$ or $P_{j,1}^{(2,d)}$ polynomials, then so does $xY_{n,k}^{(2,d)}$ but then $yY_{n,k}^{(2,d)}$ has an expansion in, respectively,  the  $P_{j,1}^{(2,d)}$ or $P_{j,0}^{(2,d)}$ polynomials.  We conclude from this and the cases listed above that $A_{n,x}^{(2,d)} \in \RR^{\dim \CV_n \times \dim \CV_n}$ is a symmetric matrix whose odd diagonals are zero since $\la xY^{(2,d)}_{n,k_1}, Y^{(2,d)}_{n,k_2} \ra_{\g_{2,d}} = 0$ if $k_1$ and $k_2$ have opposite parity. This is because, for example, if $Y^{(2,d)}_{n,k_1}$ has an expansion in the $P_{j,1}^{(2,d)}$ polynomials, then $Y^{(2,d)}_{n,k_2}$ has an expansion in the $P_{j,0}^{(2,d)}$ polynomials, which are mutually orthogonal.  Similarly, $A_{n,y}^{(2,d)} \in \RR^{\dim \CV_n \times \dim \CV_n}$ is a symmetric matrix whose even diagonals are zero since $\la y Y^{(2,d)}_{n,k_1}, Y^{(2,d)}_{n,k_2} \ra_{\g_{2,d}} = 0$ if $k_1$ and $k_2$ have the same parity. It also follows that for   $0 \leq n \leq d-1$, the structures of $B^{(2,d)}_{n,x}$ and $B^{(2,d)}_{n,y}$ are the same as those in, respectively, (\ref{eq:Bnxc1}) and (\ref{eq:Bnyc1}) but, additionally, the odd super-diagonals of $B^{(2,d)}_{n,x}$ are zero and the even super-diagonals of $B^{(2,d)}_{n,y}$  are zero. For $n \geq d$, the structures of $B^{(2,d)}_{n,x}$ and $B^{(2,d)}_{n,y}$ are the same as those in, respectively, (\ref{eq:Bnxc2}) and (\ref{eq:Bnyc2}) but the even super-diagonals of $B^{(2,d)}_{n,x}$ are zero and the odd super-diagonals of $B^{(2,d)}_{n,y}$  are zero.  To summarise, the structures of the matrices $A^{(2,d)}_{n,x}$, $A^{(2,d)}_{n,y}$, $B^{(2,d)}_{n,x}$ and $B^{(2,d)}_{n,y}$ are the same to those in (\ref{eq:Anx})--(\ref{eq:Bnyc2}), except that for $m = 2$ they have an additional checkerboard pattern.

The recurrence relations for the connection coefficients can be derived as in the case for $m = 1$ in Proposition~\ref{prop:C1forms}, but using instead (\ref{eq:C2coeffs}),
\begin{align*}
xP_{j,0}^{(2,d)} &=\b_{j-1}(w)P_{j-1,0}^{(2,d)} + \a_{j}(w)P_{j,0}^{(2,d)} + \b_{j}(w)P_{j+1,0}^{(2,d)} \\
xP_{j,1}^{(2,d)} &= \b_{j-2}(\phi w)P_{j-1,1}^{(2,d)} + \a_{j-1}(\phi w)P_{j,1}^{(2,d)} + \b_{j-1}(\phi w)P_{j+1,1}^{(2,d)}
\end{align*}
(see section~\ref{sect:multP})
and (\ref{eq:ymultyp})--(\ref{eq:ymultyp1}), which we combine with the fact that the OPs with $m = 2$ satisfy the same recurrence relations as the OPs for $m = 1$ (e.g., (\ref{eq:xY1})--(\ref{eq:yY1})), but with the additional  checkerboard pattern of $A_{n,x}^{(2,d)}$, $A_{n,y}^{(2,d)}$,  $B_{n,x}^{(2,d)}$ and  $B_{n,y}^{(2,d)}$.

\end{proof}
\section{The cases for which $C^{(2,d)}_{j,\ell(n+1,k)} = 0$}\label{Asect:C20}
\setcounter{equation}{0}

\begin{prop}\label{prop:C2struct}
For $m = 2$, $C^{(2,d)}_{j,\ell(n+1,k)} = 0$ for the following indices:
suppose $0 \leq n \leq d-2$, then
$
C^{(2,d)}_{j,\ell(n+1,k)} = 0
$
for  
\begin{equation}
j < 2n+1 \qquad \text{and } \qquad 1 \leq  k \leq n+2; \label{eq:m2jub1}
\end{equation}
if, additionally, $n$ is even, then $
C^{(2,d)}_{j,\ell(n+1,k)} = 0
$
for 
\begin{eqnarray}
\mathrm{mod}(j,2)=0 \quad  \text{or} &   \quad j > nd - 1 \quad & \text{and}  \quad k = 1, 3, \ldots, n-1  \label{eq:m2codd1} \\
\mathrm{mod}(j,2)=1 \quad  \text{or} &   \quad j > (n-2)d + 4  \quad & \text{and } \quad  k = 2, 4, \ldots, n-2 \label{eq:m2ceven1} \\
\mathrm{mod}(j,2)=1 \quad  \text{or} &   \quad j > (n-2)d + 6  \quad & \text{and } \quad  k = n  \label{eq:m2ceven2} \\
\mathrm{mod}(j,2)=0 \quad  \text{or} &   \quad j > nd  + 1  \quad & \text{and } \quad k = n + 1 \label{eq:m2codd2} \\
\mathrm{mod}(j,2)=1 \quad  \text{or} &   \quad j > nd + 2  \quad & \text{and } \quad k = n + 2 \label{eq:m2ceven3}
\end{eqnarray}
if $n$ is odd (with $0 \leq n \leq d-2$), then
$
C^{(2,d)}_{j,\ell(n+1,k)} = 0
$
for 
\begin{eqnarray*}
\mathrm{mod}(j,2)=0 \qquad  \text{or} &   \qquad j > 1 + (n-1)d \qquad & \text{and }  \qquad k = 1, 3, \ldots, n-2 \\
\mathrm{mod}(j,2)=1 \qquad  \text{or} &   \qquad j > 2 + (n-1)d \qquad & \text{and }  \qquad k = 2, 4, \ldots, n-1 \\
\mathrm{mod}(j,2)=0 \qquad  \text{or} &   \qquad j > 3 + (n-1)d \qquad & \text{and }  \qquad k =n \\
\mathrm{mod}(j,2)=1 \qquad  \text{or} &   \qquad j > 4 + (n-1)d \qquad & \text{and }  \qquad k = n+1 \\
\mathrm{mod}(j,2)=0 \qquad  \text{or} &   \qquad j > (n+1)d - 1 \qquad & \text{and }  \qquad k = n+2
\end{eqnarray*}

Suppose now that $n \geq d-1$. If $\mathrm{mod}(n+1-d,2) = 0$, then $C^{(2,d)}_{j,\ell(n+1,k)} = 0$ for 
\begin{equation}
j < 2d+d(n+1-d) \qquad \text{and} \qquad 1 \leq  k \leq d  \label{eq:m2jub2}
\end{equation}
and if $\mathrm{mod}(n+1-d,2) = 1$, then $
C^{(2,d)}_{j,\ell(n+1,k)} = 0$
for
\begin{equation}
  j < 2d - 3 + d(n+2-d) \qquad  \text{and} \qquad 1 \leq  k\leq d.   \label{eq:m2jub3}
\end{equation}

Suppose $d$ is even and $n \geq d-1$. If $\mathrm{mod}(n+1-d,2) = 0$, then $n$ is odd and  $
C^{(2,d)}_{j,\ell(n+1,k)} = 0
$ for
\begin{eqnarray*}
\mathrm{mod}(j,2)=1 \qquad  \text{or} &   \qquad j > 2 + (n-1)d \qquad & \text{and }  \qquad k = 1, 3, \ldots, d-3  \\
\mathrm{mod}(j,2)=0 \qquad  \text{or} &   \qquad j > 1 + (n-1)d \qquad & \text{and }  \qquad k = 2, 4, \ldots, d-4 \\
\mathrm{mod}(j,2)=0 \qquad  \text{or} &   \qquad j > 3 + (n-1)d \qquad & \text{and }  \qquad k = d-2 \\
\mathrm{mod}(j,2)=1 \qquad  \text{or} &   \qquad j > 4 + (n-1)d \qquad & \text{and }  \qquad k = d-1 \\
\mathrm{mod}(j,2)=0 \qquad  \text{or} &   \qquad j > (n+1)d - 1 \qquad & \text{and }  \qquad k = d
\end{eqnarray*}
and if 
 $\mathrm{mod}(n+1-d,2) = 1$, then $n$ is even and  $
C^{(2,d)}_{j,\ell(n+1,k)} = 0
$ for
\begin{eqnarray*}
\mathrm{mod}(j,2)=0 \qquad  \text{or} &   \qquad j > nd - 1 \qquad & \text{and }  \qquad k = 1, 3, \ldots, d-3 \\
\mathrm{mod}(j,2)=1 \qquad  \text{or} &   \qquad j > (n-2)d + 4 \qquad & \text{and }  \qquad k = 2, 4, \ldots, d-4\\
\mathrm{mod}(j,2)=1 \qquad  \text{or} &   \qquad j > (n-2)d + 6 \qquad & \text{and }  \qquad k = d-2\\
\mathrm{mod}(j,2)=0 \qquad  \text{or} &   \qquad j > nd  + 1 \qquad & \text{and }  \qquad k = d-1\\
\mathrm{mod}(j,2)=1 \qquad  \text{or} &   \qquad j > nd + 2 \qquad & \text{and }  \qquad k = d
\end{eqnarray*}

Suppose $d$ is odd and $n \geq d-1$. If $\mathrm{mod}(n+1-d,2) = 0$, then $n$ is even and  $
C^{(2,d)}_{j,\ell(n+1,k)} = 0
$ for
\begin{eqnarray*}
\mathrm{mod}(j,2)=1 \qquad  \text{or} &   \qquad j > (n-2)d + 4 \qquad & \text{and }  \qquad k = 1, 3, \ldots, d-4\\
\mathrm{mod}(j,2)=0 \qquad  \text{or} &   \qquad j > nd - 1 \qquad & \text{and }  \qquad k = 2, 4, \ldots, d-3\\
\mathrm{mod}(j,2)=1 \qquad  \text{or} &   \qquad j > (n-2)d + 6 \qquad & \text{and }  \qquad k = d-2\\
\mathrm{mod}(j,2)=0 \qquad  \text{or} &   \qquad j > nd  + 1 \qquad & \text{and }  \qquad k = d-1\\
\mathrm{mod}(j,2)=1 \qquad  \text{or} &   \qquad j > nd + 2 \qquad & \text{and }  \qquad k = d
\end{eqnarray*}
If $\mathrm{mod}(n+1-d,2) = 1$, then $n$ is odd and  $
C^{(2,d)}_{j,\ell(n+1,k)} = 0
$ for
\begin{eqnarray*}
\mathrm{mod}(j,2)=0 \qquad  \text{or} &   j > 1 + (n-1)d  \qquad & \text{and }  \qquad  k = 1, 3, \ldots, d-4  \\
\mathrm{mod}(j,2)=1 \qquad  \text{or} &   j > 2 + (n-1)d  \qquad & \text{and }  \qquad  k = 2, 4, \ldots, d-3 \\
\mathrm{mod}(j,2)=0 \qquad  \text{or} &   j > 3 + (n-1)d  \qquad & \text{and }  \qquad  k =d-2 \\
\mathrm{mod}(j,2)=1 \qquad  \text{or} &   j > 4 + (n-1)d  \qquad & \text{and }  \qquad  k = d-1 \\
\mathrm{mod}(j,2)=0 \qquad  \text{or} &   j > (n+1)d - 1  \qquad & \text{and }  \qquad  k = d 
\end{eqnarray*}
\end{prop}
\begin{proof}
It follows from the cases listed in (\ref{eq:y2epxs}) whether, for a given column $\mathbf{C}^{(2,d)}_{\ell(n,k)}$, the even-indexed rows or the odd-indexed rows are zero. For example, for $0 \leq n \leq d-1$, if $k$ is odd, then $C^{(2,d)}_{2j,\ell(n,k)} = 0$ for $j \geq 1$.

To determine the upper and lower bounds on the summation indices in (\ref{eq:y2epxs}), we use the same arguments as in Proposition~\ref{prop:C1struct}, which we present here with more brevity.

For $0 \leq n \leq d-2$,
\begin{equation}
\deg x^{n+2-k}y^{k-1} =\begin{cases} n+2-k+ \displaystyle{d\left(\frac{k-1}{2}\right)}, \qquad k = 1, 3, \ldots, \begin{cases}
n+ 1 & n\text{ even} \\
n+2 & n\text{ odd}
\end{cases} \\
n+3-k+ \displaystyle{d\left(\frac{k-2}{2}\right)}, \qquad k = 2, 4, \ldots, \begin{cases}
n+ 2 & n\text{ even} \\
n+1 & n\text{ odd}
\end{cases}.
\end{cases}  \label{eq:m2monomaxd}
\end{equation}  
For $n \geq d-1$ with $\mathrm{mod}(n+1-d,2) = 0$,
  \begin{equation}
\deg x^{d-k}y^{n+1+k-d} =\begin{cases} d-k+1+ \displaystyle{d\left(\frac{n+k-d}{2}\right)}, \qquad k = 1, 3, \ldots, \begin{cases}
d & n\text{ even} \\
d-1 & n\text{ odd}
\end{cases} \\
d-k+ \displaystyle{d\left(\frac{n+1+k-d}{2}\right)}, \qquad k = 2, 4, \ldots, \begin{cases}
d-1 & n\text{ even} \\
d & n\text{ odd}
\end{cases}
\end{cases} \label{eq:m2monomaxdc2}
\end{equation}
and if $n \geq d-1$ and $\mathrm{mod}(n+1-d,2) = 1$, then
\begin{equation}
\deg x^{d-k}y^{n+1+k-d} =\begin{cases} d-k+ \displaystyle{d\left(\frac{n+1+k-d}{2}\right)}, \qquad k = 1, 3, \ldots, \begin{cases}
d-1 & n\text{ even} \\
d & n\text{ odd}
\end{cases} \\
d-k+1+ \displaystyle{d\left(\frac{n+k-d}{2}\right)}, \qquad k = 2, 4, \ldots, \begin{cases}
d & n\text{ even} \\
d-1 & n\text{ odd}
\end{cases}.
\end{cases} \label{eq:m2monomaxdc3}
\end{equation}
By the same reasoning used in Proposition~\ref{prop:C1struct}, we conclude from (\ref{eq:m2monomaxd}) that if $n$ is even, with $0 \leq n \leq d-2$, then
\begin{equation}
\deg Y^{(2,d)}_{n+1,k} = \begin{cases}
dn/2  & k = 1, 3, \ldots, n-1  \\
2 + d(n-2)/2 & k = 2, 4, \ldots, n-2 \\
3 + d(n-2)/2 & k =  n \\
1 + dn/2 & k = n+1, n+2
\end{cases}.  \label{eq:Y2degsc1}
\end{equation}
For $0 \leq n \leq d-2$, it follows from  (\ref{eq:m2monomaxd}) that
\begin{align}
\deg_{\g_{2,d}}( P^{(2,d)}_{j,0} ) < n+1 \quad \text{and} \quad  \deg_{\g_{2,d}}( P^{(2,d)}_{j,1} ) < n+1 \quad \text{for} \quad j < n+1. \label{eq:P2degsc1}
\end{align}
Combining the first and second cases in (\ref{eq:y2epxs}) with (\ref{eq:Y2degsc1}) and (\ref{eq:P2degsc1}), it follows that if $n$ is even, with $0 \leq n \leq d-2$, then
\begin{equation}
Y^{(2,d)}_{n+1,k} = \sum_{j= n+1}^{dn/2}C^{(2,d)}_{2j-1,\ell(n+1,k)}P_{j,0}^{(2,d)}, \qquad k = 1, 3, \ldots, n-1 \label{eq:y2expc1} \\
\end{equation}
\begin{equation}
Y^{(2,d)}_{n+1,k} = \sum_{j= n+1}^{2+d(n-2)/2}C^{(2,d)}_{2j,\ell(n+1,k)}P_{j,1}^{(2,d)}, \qquad k = 2, 4, \ldots, n-2 \label{eq:y2expc2}  \\
\end{equation}
\begin{equation}
Y^{(2,d)}_{n+1,k} = \sum_{j= n+1}^{3+d(n-2)/2}C^{(2,d)}_{2j,\ell(n+1,k)}P_{j,1}^{(2,d)}, \qquad k = n  \\
\end{equation}
\begin{equation}
Y^{(2,d)}_{n+1,k} = \sum_{j= n+1}^{1+dn/2}C^{(2,d)}_{2j-1,\ell(n+1,k)}P_{j,0}^{(2,d)}, \qquad k = n+1   \\
\end{equation}
\begin{equation}
Y^{(2,d)}_{n+1,k} = \sum_{j= n+1}^{1+dn/2}C^{(2,d)}_{2j,\ell(n+1,k)}P_{j,1}^{(2,d)}, \qquad k = n+2.  \label{eq:y2expc5} 
\end{equation}
Note that in (\ref{eq:y2expc1}), the row indices of $C^{(2,d)}_{j,\ell(n+1,k)}$ range from $2n+1$ to $dn-1$, and hence the statements in (\ref{eq:m2jub1})  and (\ref{eq:m2codd1}) that $C^{(2,d)}_{j,\ell(n+1,k)} = 0$ for $j<2n+1$ and $j>dn-1$. Similarly, (\ref{eq:m2ceven1})--(\ref{eq:m2ceven3}) follow from, respectively, (\ref{eq:y2expc2})--(\ref{eq:y2expc5}).

The proof that $C^{(2,d)}_{j,\ell(n+1,k)} = 0$ for the remaining cases is similar. 
\end{proof}


\end{appendix}

\end{document}